\documentclass[a4paper,11pt]{article}
\usepackage[latin1]{inputenc}
\usepackage{amsfonts}
\usepackage{amsthm}
\usepackage{amsmath}
\usepackage{amsfonts}
\usepackage{latexsym}
\usepackage{amssymb}
\usepackage{dsfont}

\newtheorem{fed}{Definition}[section]
\newtheorem{teo}[fed]{Theorem}
\newtheorem*{teo*}{Theorem}
\newtheorem{lem}[fed]{Lemma}
\newtheorem{cor}[fed]{Corollary}
\newtheorem{pro}[fed]{Proposition}
\theoremstyle{definition}
\newtheorem{rem}[fed]{Remark}

\newtheorem{nota}[fed]{Notations}

\def\ga{\gamma}

\def\n0{n_{ \text{\rm \tiny o}}}

\newcommand{\IN}[1]{\mathbb {I} _{#1}}
\def\In{\mathbb {I} _n}
\def\IM{\mathbb {I} _m}
\def\suml{\sum\limits}
\def\prodl{\prod\limits}
\def\QEDP{\tag*{\QED}}

\def\bce{\begin{center}}
\def\ece{\end{center}}

\DeclareMathOperator{\FP}{FP\,}

\def\cO{{\mathcal O}} 
\def\cD{\mathcal D}

\def\subim{_{i\in \IN{n}}\,}
\def\py{\peso{and}}
\def\rk{\text{\rm rk}}
\def\noi{\noindent}
\def\cF{\mathcal F}
\def\cG{\mathcal G}
\def\QED{\hfill $\square$}
\def\EOE{\hfill $\triangle$}

\newcommand{\peso}[1]{ \quad \text{ #1 } \quad }
\def\uno{\mathds{1}}
\def\bm{\left[\begin{array}}
\def\em{\end{array}\right]}
\def\ben{\begin{enumerate}}
\def\een{\end{enumerate}}
\def\bit{\begin{itemize}}
\def\eit{\end{itemize}}
\def\barr{\begin{array}}
\def\earr{\end{array}}
\def\igdef{\ \stackrel{\mbox{\tiny{def}}}{=}\ }

\def\eps{\varepsilon}
\def\la{\lambda}

\def\N{\mathbb{N}}
\def\R{\mathbb{R}}
\def\C{\mathbb{C}}

\def\cE{\mathcal{E}}
\def\F{\mathcal{F}}
\def\G{\mathcal{G}}
\def\cH{\mathcal{H}}
\def\cK{\mathcal{K}}

\def\cP{\mathcal{P}}

\def\cS{{\cal S}}

\def\cB{{\cal B}}

\def\cV{{\cal F}}

\def\ca{\mathbf{a}}

\def\inc{\subseteq}

\def\inv{^{-1}}
\def\rai{^{1/2}}

\def\api{\langle}
\def\cpi{\rangle}

\def\ua{^\uparrow}
\def\da{^\downarrow}

\def\nuel{\nu(\la\coma t)}

 \DeclareMathOperator{\tr}{tr}
\DeclareMathOperator{\gen}{span}

\DeclareMathOperator{\leqp}{\leqslant}
\DeclareMathOperator{\geqp}{\geqslant}
\DeclareMathOperator{\convf}{Conv (\R_{\ge0})}
\DeclareMathOperator{\convfs}{Conv_s (\R_{\ge0})}

\def\RSV{\cF= \{f_i\}_{i\in \, \IN{n}}}

\newcommand{\hil}{\mathcal{H}}
\newcommand{\op}{L(\mathcal{H})}
\newcommand{\lhk}{L(\mathcal{H} \coma \mathcal{K})}

\newcommand{\posop}{L(\mathcal{H})^+}

\def\H{{\cal H}}
\def\glh{\mathcal{G}\textit{l}\,(\cH)}

\newcommand{\mat}{\mathcal{M}_d(\mathbb{C})}

\newcommand{\matsad}{\mathcal{H}(d)}

\newcommand{\matud}{\mathcal{U}(d)}

\newcommand{\matpos}{\mat^+}

\newcommand{\matinvd}{\mathcal{G}\textit{l}\,(d)}

\newcommand{\matrec}[1]{\mathcal{M}_{#1} (\mathbb{C})}

\def\beq{\begin{equation}}
\def\eeq{\end{equation}}

\def\pausa{\medskip\noi}

\begin{document}

\title{ {\bf Multiplicative Lidskii's inequalities and \\ optimal perturbations of frames}}
\author{Pedro G. Massey, Mariano A. Ruiz  and Demetrio Stojanoff\thanks{Partially supported by CONICET 
(PIP 5272/05) and  Universidad Nacional de La PLata (UNLP 11 X472).} }
\author{P. G. Massey, M. A. Ruiz and D. Stojanoff \\ {\small Depto. de Matem\'atica, FCE-UNLP,  La Plata, Argentina
and IAM-CONICET  \footnote{e-mail addresses: massey@mate.unlp.edu.ar , mruiz@mate.unlp.edu.ar , demetrio@mate.unlp.edu.ar}
}}
\date{}
\maketitle

\begin{abstract} 
In this paper we study two design problems in frame theory:
on the one hand, given a fixed finite frame $\cF$ for $\hil\cong\C^d$ we compute those dual frames $\cG$ of $\cF$ that 
are optimal perturbations of the canonical dual frame for $\cF$ under certain restrictions on the norms of the elements of $\cG$. On the other hand, for a fixed finite frame $\cF=\{f_j\}_{j\in\In}$ for $\hil$ 
we compute those invertible operators $V$ such that $V^*V$ is a perturbation of the identity and such that 
the frame $V\cdot \cF=\{V\,f_j\}_{j\in\In}$ - which is equivalent to $\cF$ - is optimal among such perturbations of $\cF$. In both cases, optimality is measured with respect to submajorization of the eigenvalues of the frame operators. Hence,  our optimal designs are minimizers of a family of convex potentials that include the frame potential and the mean squared error. 
The key tool for these results is a multiplicative analogue of Lidskii's inequality in terms of log-majorization and a characterization of the case of equality.
 \end{abstract}

\def\coma{\, , \, }



\date{}

\def\coma{\, , \, }

\tableofcontents

\section{Introduction} 

A finite frame for $\hil\cong\C^d$ is a sequence $\F=\{f_j\}_{j\in\In}$ that spans $\hil$, where $\In=\{1,\ldots,n\}$ (for a detailed exposition on frames and several recent research topics within this theory see \cite{FF,Chrisbook} and the references therein). Given a frame $\F=\{f_j\}_{j\in\In}$, a sequence $\cG=\{g_j\}_{j\in \In}$ is called a dual frame for $\cF$ if for every $f\in\hil$ the following reconstruction formulas hold: $$f=\sum_{j\in\In}\langle f,\,g_j\rangle\ f_j \ \text{ and  } \ f=\sum_{j\in\In}\langle f,\,f_j\rangle\ g_j \ .$$
Hence, frames provide a (possibly redundant) linear-encoding scheme for vectors in $\hil$.

Let $\F=\{f_j\}_{j\in\In}$ be a frame for $\hil$ and let $\cD(\cF)$ denote the set of dual frames for $\cF$.
There is a distinguished dual called the canonical dual of $\cF$, denoted $\cF^\#\in\cD(\cF)$, which is a natural choice in several ways. But in case $n>d$ it is well known that $\cD(\cF)$ has a rich structure (this last fact is one of the main advantage of frames over bases $\cB=\{v_j\}_{i\in\IN{d}}$ for which $\cD(\cB)$ becomes a singleton).
Thus, in applied situations, the structure of $\cD(\cF)$ can be exploited to obtain numerically stable encoding-decoding schemes derived from the dual pair $(\cF,\cG)$, for some choice of dual frame $\cG\in\cD(\cF)$ beyond $\cF^\#$; this 
is the starting point of the so-called (optimal) design problems for dual frames (see \cite{LeHanagre,LoHanagre,MRS2,MRS3,MRS4}). 

Another research topic in frame theory is the design of (optimal) stable configurations of vectors (frames) under certain restrictions. Typically, the stability of a frame $\cF$ is measured in terms of the spread the eigenvalues of 
the positive semidefinite operator $S_\cF=\sum_{j\in\In} f_j\otimes f_j$. One of the most important examples of such a measure is the frame potential of $\cF$, denoted by $\FP(\cF)$, introduced in \cite{BF}; explicitly, for a sequence $\cF=\{f_j\}_{j\in\In}$ then $$\FP(\cF)=\sum_{j,\,k\in\In}|\langle f_j,\,f_k\rangle |^2=\tr(S_\cF^2)\ .$$
In \cite{BF,CKFT} it is shown that minimizers of the frame potential, within convenient sets of frames, have many nice structural features and are optimal in several ways. 
Recently, there has also been interest in the so-called mean squared error of $\cF$, denoted MSE$(\cF)$, given by MSE$(\cF)=\tr(S_\cF^{-1})$ (see \cite{FMP,MR,MRS12}). 

It turns out that there is an structural measure of optimality, called sub-majorization, that allows to deal with both the frame potential and the mean squared error. This pre-order relation, defined between eigenvalues of frame operators, has proved useful in explaining the structure of minimizers of convex potentials (see \cite{MR}). Sub-majorization has also been useful in obtaining the structure of optimal vector configurations as well (see \cite{MRS4,MRS pre}). In turn, sub-majorization relations imply a family of tracial inequalities in terms of convex functions, that contain the frame potential and mean squared errors. We point out that these tracial inequalities have interest in their own right and collectively characterize sub-majorization.

In this paper we consider two optimal design problems, where optimality is measured in terms of sub-majorization.
On the one hand, given a fixed frame $\cF$, we consider the problem of designing a dual frame $\cG=\{g_j\}_{j\in\In}\in \cD(\cF)$ such that $\sum_{j\in\In}\|g_j\|^2\geq t$ for some appropriate $t>0$, the distance between $\cG$ and $\cF^\#$ is controlled by some $\epsilon>0$ and such that the spread of the eigenvalues of $S_\cG$ is minimal (with respect to sub-majorization) among the eigenvalues of $S_{\cG'}$ for all such $\cG'$. Thus, we refine the analysis of the optimal design problem 
for dual frames obtained in \cite{MRS4} 
(see Section \ref{el prob1} for a detailed description of this problem and some further motivations). Since we keep control of the distance of these frames to the canonical dual, we consider this optimal solution as a perturbation of the canonical dual (that improves some of its numerical features). On the other hand, given a fixed frame $\cF=\{f_j\}_{j\in\In}$ for $\hil\cong\C^d$ with $n>d$ 
we consider the problem of designing an invertible operator $V$ acting on $\hil$ such that $V^*V$ is a perturbation of the identity operator i.e. such that $V$ is nearly a unitary operator (see Section \ref{sec opt cohe pert} for a detailed description of this problem and further motivations) and such that, if we denote by $V\cdot \cF=\{V\,f_j\}_{j\in \In}$ then, 
the spread of the eigenvalues of $S_{V\cdot \cF}$ is minimal (with respect to sub-majorization) among the eigenvalues of $S_{V'\cdot \cF}$ for all such $V'$. Notice that a frame $\cG$ for $\hil$ can be written as $V\cdot \cF$ for some invertible operator $V$ acting on $\hil$ if and only if the frames $\cF$ and $\cG$ have the same linear relations. Two such frames are called equivalent (see \cite{Bal}): hence, we search for frames $V\cdot \cF$ that are equivalent to $\cF$ (in the previous sense) and such that they improve some numerical features of $\cF$. Moreover, since $V^*V$ is a perturbation of the identity operator then $V\cdot \cF$ is a perturbation of $\cF$ - up to unitary equivalence.

In order to tackle both problems above, we introduce abstract models for them within the framework of matrix analysis. Although the frame problems seem unrelated, it turns out that the abstract model for the design of optimal duals plays a crucial role in the analysis of the abstract model for the perturbations by equivalent frames.

The key tools for these results are the multiplicative analogue of Lidskii's inequality in terms of log-majorization 
obtained by Li and Mathias in \cite{LiMat}, and a characterization of the case of equality (see Section \ref{sec append}). 
We also use the 
optimality results 
proved in  \cite{MRS12}, which are based in 
the additive case for Lidskii's inequality and the case of equality, studied in the appendix of that paper.

The paper is organized as follows. In Section \ref{sec 2}, after setting the general notations used throughout the paper, we describe the basic framework of finite frames that we shall need, together with a brief description of general convex potentials. We also include a description of sub-majorization and log-majorization which are two notions from matrix theory.
In Section \ref{sec the problems} we give a detailed description of the two problems in frame theory that 
we consider in this note, using the notations and terminology from Section \ref{sec 2}. In Section \ref{sec 4} we first introduce an abstract model for the design of optimal dual frames with restrictions, and apply tools from matrix analysis to obtain optimality results; we then apply these results to the original frame problem. Similarly, in Section \ref{sec 5 opt cohe pert} we first analyze an abstract model for the design of optimal perturbations of a frame by equivalent frames and then apply the results of the abstract model to the original frame problem. Finally, in Section \ref{sec append} we develop some aspects of the multiplicative Lidskii's inequality with respect to log-majorization which are needed for the analysis of the abstract model in Section \ref{sec 5 opt cohe pert}.

\section{Preliminaries}\label{sec 2}
In this section we describe the basic notions that we shall consider throughout the paper. 

\subsection{General notations.}

Given $\cH \cong \C^d$  and $\cK \cong \C^n$, we denote by $\lhk $ 
the space of linear operators $T : \cH \to \cK$. 
If $\cK = \cH$ we denote by $\op = L(\cH \coma \cH)$, 
by $\glh$ the group of all invertible operators in $\op$, 
 by $\posop $ the cone of positive operators and by
$\glh^+ = \glh \cap \posop$. 
If $T\in \op$, we  denote by   
$\|T\|$ its operator (spectral) norm,
by $\rk\, T= \dim R(T) $  the rank of $T$,
and by $\tr T$ the trace of $T$.

\pausa
If $W\inc \cH$ is a subspace we denote by $P_W \in \posop$ the orthogonal 
projection onto $W$. 
Given $x\coma y \in \cH$ we denote by $x\otimes y \in \op$ the rank one 
operator given by 
\beq \label{tensores}
x\otimes y \, (z) = \api z\coma y\cpi \, x \peso{for every} z\in \cH \ .
\eeq
Note that if $\|x\|=1$ then $x\otimes x = P_{\gen\{x\}}\,$.

\pausa 
By fixing orthonormal bases (ONB) 
of the Hilbert spaces involved, we shall identify operators with 
matrices, using the following notations: 
by $\matrec{n,d} \cong L(\C^d \coma \C^n)$ we denote the space of complex $n\times d$ matrices. 
If $n=d$ we write $\mat = \matrec{d,d}$ ;   
$\matsad$ is the $\R$-subspace of selfadjoint matrices,  
$\matinvd$ the group of all invertible elements of $\mat$, $\matud$ the group 
of unitary matrices, 
$\matpos$ the set of positive semidefinite
matrices, and $\matinvd^+ = \matpos \cap \matinvd$. 

\pausa
Given $m \in \N$ we denote by $\IM = \{1, \dots , m\} \inc \N$ and 
$\uno = \uno_m  \in \R^m$ denotes the vector with all its entries equal to $1$. 
For a vector $x\in \R^m$ we denote by 
$\tr \, x = \sum_{i\in \IN{d}} \, x_i$ 
and by $x^\downarrow$ (resp. $x^\uparrow$) the rearrangement
of $x$ in  decreasing (resp. increasing) order. We denote by
$(\R^m)^\downarrow = \{ x\in \R^m : x = x^\downarrow\}$ the set of downwards ordered vectors, and similarly $(\R^m)\ua =\{x\in \R^m:\ x=x\ua\}$.

\pausa
Given $S\in \matpos$, we write $\la(S) = \la\da(S)\in 
(\R_{\geq 0}^d)^\downarrow$ the 
vector of eigenvalues of $S$ - counting multiplicities - arranged in decreasing order. 
Similarly we denote by $\la\ua(S) \in (\R_{\geq 0}^d)\ua$ the reverse ordered 
vector of eigenvalues of $S$. 

\subsection{Basic framework of finite frames }\label{basic}

Let $d, n \in \N$, with $d\le n$. Fix a Hilbert space $\hil\cong \C^d$. 
A family $ \RSV \in  \cH^n $  is an 
 frame for $\cH$  if there exist constants $A,B>0$ such that
\beq\label{frame defi} A\ \|x\|^2\leq \sum_{i\in\In} |\left \langle x \, , f_i\right \rangle|^2\leq B \ \|x\|^2 \peso{for every} x\in \hil \ .
\eeq
The { optimal frame bounds}, denoted by  $A_\cF, B_\cF$ are the optimal constants in Eq. \eqref{frame defi}. If $A_\cF=B_\cF$ we call $\cF$ a tight frame.
Since $\dim \hil<\infty$, a family  $\RSV$  is an 
frame if and only if $\gen\{f_i: i \in \In \} = \cH$.  

\pausa
Given $\RSV \in  \cH^n $, we consider its 
{ analysis} operator $T_\cV \in L(\hil\coma\C^n)$, given 
by 
 \beq \ T_\cV\, x= \big( \,\api x \coma f_i\cpi\,\big)  \subim 
\ , \peso{for every} x\in \cH \ .
\eeq
  Its adjoint $T_\cV^* \in L(\C^n\coma \cH)$ is called the { synthesis} operator and it is given by 
 $T_\cV ^* \, \ca =\sum_{i\in \, \IN{n}} a_i\, f_i$ 
for every $\ca = (a_i)_{i\in\In}\in \C^n$. 
With the notations of \eqref{tensores}, the frame operator of $\cV$ is 
$$
\barr{rl}
S_\cV &= T_\cV^*\  T_\cV = \sum_{i \in \In} f_i \otimes f_i 
\in \posop\ . \earr
$$
Notice that, if 
$\RSV \in  \cH^n $ then  
$\api S_\cV \, x\coma x\cpi \, = \sum\subim \, 
 \big|\, \api x \coma f_i\cpi \, \big|^2$ for every 
 $x\in \cH$. Hence, $\cF$ is a frame if and only if 
$S_\cF\in \glh^+$ and in this case
$A_\cV \, \|x\|^2 \, \le \, \api S_\cV \, x\coma x\cpi 
 \,  \le \, B_\cV \, \|x\|^2$ for every $ x\in \cH $.   
Therefore, $A_\cV   =\la_{\min} (S_\cV) = \|S_\cV\inv \| \inv$ and $ 
\la_{\max} (S_\cV) = \|S_\cV \| = B_\cV \,$.
Moreover, $\cV $ is tight if and only if $S_\cV = 
\frac{\tau}{d}  \, I_\H\,$, where $\tau = 
\tr S_\cV = \sum\subim \|f_i\|^2 \,$.

\pausa Let $\cF=\{f_i\}_{i\in\In}\in\hil^n$ be a frame for $\hil$. A family $\cG=\{g_i\}_{i\in\In}$ is said to be a {  (alternate) dual} of $\cF$ if $$ f=\sum_{i\in\In} \langle f,g_i\rangle \ f_i \ , \ \ \text{ for every } f\in\hil\ .
$$ It is easy to see that $\cG$ is a dual of $\cF$ iff $T_\cF^*\, T_\cG=I_\hil$. Hence, in that case $\cG$ is a frame and $\cF$ is a dual of $\cG$ and we say that $(\cF,\cG)$ is a dual pair of frames for $\hil$. We shall consider $$\cD(\cF)=\{ \cG: \cG \text{ is a dual frame of } \cF\ \}\subset \hil ^n\ . $$
The so-called canonical dual of $\cF$, denoted $\cF^\#$, is given by $\cF^\#=\{S_\cF^{-1}f_i\}_{i\in\In}$. It is straightforward to check that $T_{\cF^\#}=T_\cF S_\cF^{-1}$ and then $T_{\cF}^* T_{\cF^\#}=S_{\cF}\,S_{\cF}^{-1}=I_\hil$ so $\cF^\#\in\cD(\cF)$. The canonical dual is a distinguished dual since it possesses several minimality properties. Nevertheless, notice that whenever $\dim\hil=d<n$ (i.e. whenever the frame $\cF=\{f_i\}_{i\in\In}$ is a redundant set of linear generators) then $\cD(\cF)$ has infinitely many elements and it turns out to have a very rich  structure. This is one of the main advantages of the redundant frame $\cF$ over a (not necessarily orthonormal) basis $\cB=\{v_i\}_{i\in\IN{d}}$ in $\hil$, since the set of duals of the latter has only one element, namely $\cD(\cB)=\{\cB^\#\}$.

\pausa 
In their seminal work \cite{BF}, Benedetto and Fickus introduced a functional defined (on unit norm frames), the so-called frame potential, given by 
$$ 
\barr{rl}
\FP(\{f_i\}_{i\in \In} )& 
=\sum_{i,\,j\,\in \In}|\api f_i\coma f_j \cpi |\,^2\ .
\earr
$$ 
One of their major results shows that tight unit norm frames - which form an important class of frames because of their simple reconstruction formulas - can be characterized as (local) minimizers of this functional among unit norm frames. Since then, there has been interest in (local) minimizers of the frame potential within certain classes of frames, since such minimizers can be considered as natural substitutes of tight frames. 
Notice that, given $\cF=\{f_i\}_{i\in \IN{n}}\in  \cH^n$ then $\FP(\cF)=\tr \, S_\cF^2 
=\sum_{i\in \IN{d}}\lambda_i(S_\cF)^2$. Recently, there has been interest in the structure of minimizers of other potentials such as the so-called mean squared error (MSE) given by MSE$(\cF)=\tr(S_\cF^{-1})$.  

\pausa
These remarks have motivated 
the analysis of the structure of minimizers of general convex potentials:

\begin{fed}\label{pot generales}\rm
Let us denote by 
$$
\convf = \{ 
h:[0 \coma \infty)\rightarrow [0 \coma \infty): h   \ \mbox{ is a convex function } \ \} 
$$  
and $\convfs = \{h\in \convf : h$ is strictly convex $\}$. 
Following \cite{MR} we consider the 
(generalized) {convex potential} $P_h$ associated to $h\in \convf$, given by
$$
\barr{rl}
P_h(\cF)&=\tr \, h(S_\cF) = \sum_{i\in \IN{d}}h(\lambda_i(S_\cF)\,) \peso {for} 
\cF=\{f_i\}_{i\in \IN{n}}\in  \cH^n \ , \earr
$$
where the matrix $h(S_\cF)$ is defined by means of the usual functional calculus in $\posop$. \EOE
\end{fed}

\pausa
In order to deal with these general convex potential we consider the notions of submajorization and log-majorization in the next section.

\subsection{Submajorization and log-majorization} \label{subsec prelims mayo}

Next we briefly describe majorization and log-majorization, two notions from matrix analysis theory that will be used throughout the paper. For a detailed exposition on these relations see \cite{Bat}.

\pausa 
 Given $x,\,y\in \R^d$ we say that $x$ is
{ submajorized} by $y$, and write $x\prec_w y$,  if
$$
\suml_{i=1}^k x^\downarrow _i\leq \suml_{i=1}^k y^\downarrow _i \peso{for every} k\in \mathbb I_d \ .
$$  
If $x\prec_w y$ and $\tr x = \sum_{i=1}^dx_i=\sum_{i=1}^d y_i = \tr y$,  we say that $x$ is
majorized by $y$, and write $x\prec y$. 

\pausa
On the other hand we write 
$x \leqp y$ if $x_i \le y_i$ for every $i\in \mathbb I_d \,$.  It is a standard  exercise 
to show that $x\leqp y \implies x^\downarrow\leqp y^\downarrow  \implies x\prec_w y $. 
Our interest in majorization is motivated by the
relation of this notion with tracial inequalities 
for convex functions. 
Indeed, given $x,\,y\in \R^d$ and  $f:I\rightarrow \R$ a 
convex function defined on an interval $I\inc \R$ such that 
$x,\,y\in I^d$,  then (see for example \cite{Bat}): 
\ben 
\item If one assumes that $x\prec y$, then 
$ 
\tr f(x) \igdef\suml_{i\in\IN{d}}f(x_i)\leq \suml_{i\in\IN{d}}f(y_i)=\tr f(y)\ .
$
\item If only $x\prec_w y$,  but the map $f$ is also increasing, then  still 
$\tr f(x) \le \tr f(y)$. 
\item If $x\prec_w y$ and $f$ is a strictly convex function such 
that $\tr \,f(x) =\tr \, f(y)$ then there exists a permutation $\sigma$ 
of $\IN{d}$ such that $y_i=x_{\sigma(i)}$ for $i\in \IN{d}\,$. 
\een

\begin{rem}\label{rem doublystohasctic}
Majorization between vectors in $\R^d$ is intimately related with the 
class of doubly stochastic $d\times d$ matrices, denoted by DS$(d)$. Recall that a 
$d\times d$ matrix $D \in $ DS$(d)$   if it has non-negative entries and each row sum and column sum equals 1. 

\pausa
It is well known (see \cite{Bat}) that given $x\coma y\in \R^d$ then $x\prec y$ 
if and only if there exists $D\in$ DS$(d)$ such that $D\, y=x$. As a consequence 
of this fact we see that if $x_1\coma y_1\in \R^r$ and $x_2 \coma y_2\in \R^s$ 
are such that $x_i\prec y_i\, $, $i=1\coma 2$, 
then $x=(x_1\coma x_2)\prec y=(y_1\coma y_2)$ in $\R^{r+s}$.
Indeed, if $D_1$ and $D_2$ are the doubly stochastic matrices corresponding the previous majorization relations then $D=D_1\oplus D_2\in$ DS$(r+s)$ is such that $D\, y=x$. 
\EOE
\end{rem}

\pausa
Log-majorization between vectors in $\R_{\geq 0}^d$ is a multiplicative analogue of majorization in $\R^d$. Indeed, given $x,\, y\in \R_{\geq 0}^d$ we say that $x$ is { log-majorized} by $y$, denoted $x\prec_{\log} y$, if 
\beq\label{mayolog}
\prodl_{i=1}^k x^\downarrow _i\leq \prodl_{i=1}^k y^\downarrow _i \peso{for every} k\in \mathbb I_{d-1} \peso{and}  
\prodl_{i=1}^d x^\downarrow _i= \prodl_{i=1}^d y^\downarrow _i \ .
\eeq  
Our interest in log-majorization is also motivated by the relation of this notion with tracial inequalities for convex functions. It is known (see \cite{Bat}) that $$ \peso{if} x,\, y\in \R_{\geq  0}^d \ , \ \  x\prec_{\log} y  \ \ \Rightarrow \ \ x\prec_w y\ .$$ Hence, if $x,\, y\in \R_{\geq 0}^d$ are such that $x\prec_{\log} y$ then for every convex and increasing function $f:(0,\infty)\rightarrow R$ we get that $\tr(f(x))\leq \tr(f(y))$.

\section{Two problems in frame theory}\label{sec the problems}

In this section we present a detailed description of the two frame problems together with some
further motivations, using the notations and terminology from Section \ref{sec 2}.

\subsection{Optimal perturbation of the canonical dual frame with restrictions}\label{el prob1}

Consider a fixed frame $\F=\{f_j\}_{j\in\In}$ for $\hil\cong\C^d$, with $n>d$. Then, the set of dual frames $\cD(\cF)$ has a rich structure. It is known that if $\cG\in \cD(\cF)$ then $ S_\cG\geq S_{\cF^\#}=S_\cF^{-1}$, with respect to the operator order. This strong inequality explains several optimality properties of the canonical dual frame $\cF$. For example, if $\cG=\{g_i\}_{i\in\In}$ and $\cF^\#=\{f^\#_i\}_{i\in\In}$ we have that 
$$\sum_{i\in\In}\|f^\#_i\|^2=\tr(S_{\cF^\#})\leq \tr(S_{\cG})=\sum_{i\in\In}\|g_i\|^2 \, , $$ with equality if and only if $\cG=\cF^\#$.

\pausa
Nevertheless, in applied situations, it is desired to consider numerically stable encoding-decoding schemes derived from the dual pair $(\cF,\cG)$, for some choice of dual frame $\cG\in\cD(\cF)$. 
A possible way out of this situation is as follows: for $t>\tr(S_{\cF^\#})$ consider 
$$\cD_t(\cF)=\{\cG=\{g_i\}_{i\in\In}\in \cD(\cF):\ \sum_{i\in\In}\|g_i\|^2\geq t\}$$ 
and search for optimal duals within $\cD_t(\cF)$ with respect to some measure of optimality (e.g. minimizers of the condition number). It turns out (see \cite{MRS4}) that there exists a distinguished class $\cO\cD_t(\cF)\subset\cD_t(\cF) $
such that for every $\cG^o\in \cO\cD_t(\cF)$ and every $\cG\in \cD_t(\cF)$ we have the majorization relation $\lambda(S_{\cG^o})\prec_w \lambda(S_\cG)$. This last fact implies several optimality properties of the class $\cO\cD_t(\cF)$.

\pausa
Still, in the search for optimal alternative duals for $\cF$, there are some properties of the canonical dual frame $\cF^\#$ that we may want to retain. In order to preserve some of the minimal features of the canonical dual frame and yet search for numerically stable alternative duals (which are possibly best suited for practical purposes) we introduce the following class of dual frames: set $m=2d-n$, let $t> \tr(S_{\F^\#})$ and $\eps>0$ be such that $t-\tr(S_{\F^\#})\leq \min\{(d-m),\,d\}\cdot \eps^2$ and define
$$\cD_{(t \coma\eps)}(\cF)=\{\G=\{g_j\}_{j\in\In}\in \cD(\F): \ \sum_{j\in\In}\|g_j\|^2\geq t \ , \ \|T_\G-T_{\F^\#}\|\leq \eps \}\ . $$ 
Hence, we search for optimal duals for $\cF$ within $\cD_{(t \coma\eps)}(\cF)$. As a criteria for optimality,
 following \cite{MRS4} we search for $\prec_w$-minimizers of the eigenvalues of the frame operators $S_\cG$ for $\cG\in \cD_{(t \coma\eps)}(\cF)$. Hence, we consider the associated set 
$$
\cS(\cD_{(t \coma \eps)}(\cF))\igdef\{S_\G: \ \G\in \cD_{(t \coma \eps)}(\cF)\} \inc\matpos\ .
 $$
As we shall see, there exist $\prec_w$-minimizers within $\cS(\cD_{(t \coma \eps)}(\cF))$; moreover, their spectral and geometrical features can be explicitly computed. We point out that the structure of optimal duals 
depends both on the norm restriction (of the frame elements of $\cG$) and on the operator norm distance restriction. As a first step in our analysis, we obtain an explicit representation of the frame operators of the elements of $\cD_{(t \coma \eps)}(\cF)$.

\begin{pro}\label{carac de ops frame duals}
Let $\F=\{f_j\}_{j\in\In}$ be a frame for $\C^d$ and set $m=2d-n$. Let $t> t_0= \tr(S_{\F^\#})$ and $\eps>0$ 
be such that $ t-t_0 
\leq \min\{(d-m),\, d\}\cdot \eps^2$. Then
$$ \cS(\cD_{(t \coma \eps)}(\cF))= \{S_{\cF^\#}+B :\ B\in\matpos\ , \ \tr(B)\geq t-t_ 0 
\ , \ \|B\|\leq \eps^2\ , \ \rk(B)\leq d-m \ \} \ .$$
\end{pro}
\begin{proof}
Let $\G\in\cD_{(t \coma \eps)}(\cF)$ and notice that then 
$$
T_\G^*T_\F= T_{\F^\#}^*T_\F=I_\hil
\implies (T_\G^* - T_{\F^\#}^*)T_\F=0\ .
$$ 
Hence, if we let $A=T_\G - T_{\F^\#}:\C^d\rightarrow \C^n$ then $A^*\,T_\F=0$ which implies that $T_\F^*A=0$ and hence $R(A)\inc \ker (T_\F^*)=\ker (T_{\cF^\#}^*)$. Therefore, $T_\G=T_{\F^\#}+A$ and $$ T_\G^*\,T_\G= T_{\F^\#}^* T_{\F^\#} + T_{\F^\#}^* A + A^* T_{\F^\#} + A^*A= T_{\F^\#}^*T_{\F^\#}+A^*A$$ that is, $S_\G=S_{\F^\#}+B$, where $B=A^*A\in\matpos$. Since $R(A)\inc \ker (T_\F^*)$ then $\rk(B)=\rk(A)\leq n-d=d-m$ and $\|B\|=\|A^*A\|=\|T_\G-T_{\F^\#}\|^2\leq \eps^2$. Notice that $\tr(S_\G)=\sum_{j\in\In}\|g_j\|^2\geq t$ which shows that $\tr(B)=\tr(S_\cG)-\tr(S_{\cF^\#})\geq t-\tr(S_{\cF^\#})$.
Incidentally, notice that 
the existence of such a $B$ implies the restriction on $t$ of the statement, since 
\beq\label{t y eps}
\tr(B)\leq \rk(B)\cdot \|B\| 
\implies 
t-\tr(S_{\cF^\#})\leq \tr(B)\leq \min\{(d-m),\,d\}\cdot \eps^2 \ .
\eeq
In order to show the converse inclusion, let $B\in\matpos$ be such that $\tr(B)\geq t-\tr(S_{\cF^\#})$, $\|B\|\leq \eps^2$ and $\rk(B)\leq d-m=n-d=\dim \ker T_{\cF}^*$. Then, we can factor $B=A^*A$ where $A:\C^d\rightarrow \C^n$ is such that $R(A) \subseteq \ker (T_{\cF}^*)$, so that $T_{\cF}^*A=0$ and $A^*T_{\cF}=0$. Let $\cG=\{(T_{\cF^\#}+A)^* e_i \}_{i\in\In}$, where $\{e_i\}_{i\in\In}$ denotes the canonical basis of $\C^n$. Thus, $T_\cG=T_{\cF^\#}+A$ and hence $T_\cG^*T_\cF=I_\hil$, $S_\cG=S_{\cF^\#}+B$; in particular, $\sum_{i\in\In} \|g_i\|^2=\tr(S_\cG)=\tr(S_{\cF^\#})+\tr(B)\geq t$ and  $\|T_\cG-{T_\cF^\#}\|=\|A\|=\|B\|^{1/2}\leq \eps$. Therefore, $\cG\in \cD_{(t \coma \eps)}(\cF)$ is such that $S_{\cG}=S_{\cF^\#}+B$.
\end{proof}

\begin{rem}
With the notations of Proposition \ref{carac de ops frame duals}, by Eq. \eqref{t y eps} and a 
straightforward construction of a matrix $B$ with the required parameters,  we see that the relation
$$
t-\tr(S_{\F^\#})\leq \min\{(d-m),\,d\}\cdot \eps^2
$$ 
between the parameters $t$ and $\eps$ 
is necessary and sufficient for $\cD_{(t \coma\eps)}(\cF)\neq \emptyset$. \EOE 
\end{rem}

\subsection{Optimal perturbations by equivalent frames}\label{sec opt cohe pert}

Fix a frame $\cF=\{f_i\}_{i\in\In}$ for $\C^d$, with $n>d$. Hence, $\cF$ is a redundant family of linear generators for $\C^d$; in other words, $T_\cF:\C^d\rightarrow \C^n$ is an injective transformation such that $R(T_\cF)\subset \C^n$ is a proper subspace. These last facts can be used to develop some simple linear tests in order to check whether a sequence $\ca=(a_i)_{i\in\In}\in\C^n$ is the sequence of frame coefficients $T_\cF(f)=(\langle f,f_i\rangle )_{i\in\In}$ for some (unique) $f\in\hil$ or whether it has been corrupted (e.g. due to noise in the communication channel). 
Indeed, given a linear relation $\sum_{i\in\In}\alpha_i\, f_i=0$ 
of the family $\cF$, 
 we get a linear test $\varphi[(\langle f,f_i\rangle)_{i\in\In}]=\sum_{i\in\In}\overline{\alpha_i} \, \langle f,f_i\rangle=0$ for the sequence $\ca$; moreover, we can consider a complete set of tests $\varphi_1,\ldots,\varphi_{n-d}\in(\C^n)^*$ in order the check whether the sequence $\ca$ lies in $R(T_\cF)$, based on the linear relations of the family $\cF$.

\pausa
Hence, the linear relations among the elements of $\cF$ play an important role in this context. 
 On the other hand, the numerical stability of the frame $\cF$ is also an important feature in practice. Hence, there are situations in which we want to improve the stability of the frame $\cF$ (measured in terms of the spread the eigenvalues of its frame operator) while preserving the linear relations among the frame elements. It is well known that if a family $\cG=\{g_i\}_{i\in\In}$ has the same linear relations as $\cF$  then there exists an invertible linear operator $V\in \cG l(\hil)$ such that $\cG=V\cdot \cF=\{Vf_i\}_{i\in\In}$. In this case, following \cite{Bal} we say that $\cG=V\cdot \cF=\{Vf_i\}_{i\in\In}$ and $\cF$ are equivalent frames.

\pausa
As a first step, we can search for an invertible operator $V\in \cG l(\hil)$ such that 
the frame $V\cdot \cF$ is optimal with respect to the spread of the eigenvalues of its frame operator.
It turns out that the solution to this (unrestricted) problem is $V=S_{\cF}^{-1/2}$ so that 
$V\cdot \cF=\{S_\cF^{-1/2}f_i\}_{i\in\In}$ is the associated Parseval frame (with minimal spread of its eigenvalues). This solution, although optimal, might lie {\it away} from the original frame $\cF$ in the sense that it is a strong deformation of the space $\hil$. In case we want to consider rather {\it perturbations} of $\cF$, preserving their linear relations and yet improving its numerical performance, we can search for invertible operators $V$ such that the equivalent frame $V\cdot \cF$ is optimal in the previous sense, under some restrictions on $V$. Hence, given a frame $\cF=\{f_i\}_{i\in\In}$ we introduce the following set of controlled perturbations by equivalent frames: given $0<\delta<1$ and $s\in[(1-\delta)^d,(1+\delta)^d]$ then consider 
\beq \label{defi cp}
\cP_{(s \coma \delta)}(\cF)\igdef\{V\cdot\cF=\{Vf_i\}_{i\in\In}:\ V\in\cG l(\hil)\ , \ \|V^*V-I\|\leq \delta \ , \ \det(V^*V)\geq s\ \}\ . 
\eeq
Our main problem is to compute the structure optimal perturbations $V\cdot\cF\in \cP_{(s \coma \delta)}(\cF)$, in the sense that they minimize the spread of the eigenvalues of the frame operators within this class.	
Hence, in order to deal with this problem we also introduce 
\beq\label{defi scp} 
\cS(\cP_{(s \coma \delta)}(\cF))\igdef\{S_{\cG}: G\in\cP_{(s \coma \delta)}(\cF)\}
\eeq
 and it is straightforward to check that 
\beq \label{ident scp1}
\cS(\cP_{(s \coma \delta)}(\cF))=\{V^*S_\cF V:\ V\in\cG l(\hil)\ , \ \|V^*V-I\|\leq \delta \ , \ \det(V^*V)\geq s\ \}\ . 
\eeq
As we shall see, there exist $\prec_w$ minimizers within $\cS(\cP_{(s \coma \delta)}(\cF))$; moreover, their
spectral and geometrical features can be explicitly computed.

\section{Optimal perturbation of the canonical dual frame}\label{sec 4}

In this section, given a frame $\cF$, we consider a matrix model for the design of optimal perturbations of the canonical dual frame $\cF^\#$. In Section \ref{sec analisis dual}, using tools from matrix analysis, we show that there are optimal solutions within our abstract model. Then, in Section \ref{aplic dual} we apply these results to the initial frame design problem.

\subsection{A matrix model for $\cS(\cD_{(t \coma \eps)}(\cF))$}\label{sec analisis dual}

In order to deal with the problem of existence of $\prec_w$ minimizers in $\cS(\cD_{(t \coma \eps)}(\cF))$ for a frame $\cF$, we introduce the following set motivated by Proposition \ref{carac de ops frame duals}: 
given $S\in\matpos$, $t>t_0\igdef\tr(S)$, $\eps>0$ and $m\in\mathbb Z$ with $m\leq d-1$ such that $t-t_0\leq \min\{(d-m),\, d\}\cdot \eps$, we define
\beq \label{defi ut}
U_{(t \coma \eps)}(S,m)\igdef\{S+B:\ B\in\matpos\ , \ \tr(B)\geq t-t_0\ , \ \|B\|\leq \eps\ , \ \rk(B)\leq d-m \}\ .
\eeq
Hence, Proposition \ref{carac de ops frame duals} states (using the notations in that result) that 
$$\cS(\cD_{(t \coma \eps)}(\cF))=U_{(t \coma \eps^2)}(S_{\cF^\#}\coma m)\ . $$
In order to deal with the spectral structure of optimal elements in $U_{(t \coma \eps)}(S,m)$ the following (rather naive) vector model will be useful: 
given $\la\in(\R^d)\da$, $t>t_0\igdef\tr(\la)$, 
$\eps>0$ and $m\in\mathbb Z$ with $m\leq d-1$ such that $t-t_0\leq\min\{(d-m),\,d\}\cdot \eps$ we define $$
\Lambda_{(t \coma \eps)}(\la  \coma m)=\{ \la+\mu:\ 0\leq \mu_i\leq \eps \ , \ \tr \, \mu \geq t-t_0 \ , \ \text{supp}(\mu)\leq d-m\}\inc\R^d_{\geq 0} \ 
$$  where supp$(\mu)=\|\mu\|_0=\#\{j\in\IN{d}: \ \mu_j\neq 0\}$.

\begin{rem}\label{resultados previos1}
In \cite{MRS4} (also see \cite{MRS12,MRS pre}) we considered a simpler version of the set 
$\Lambda_{(t \coma \eps)}(\la  \coma m)$ defined above. Indeed, given $\lambda=(\la_i)_{i\in\IN{d}}\in (\R^d)^\downarrow$, $t\geq t_0=\tr(\la)$ we considered the set 
$$
\Lambda_t(\la)=\{\la+\mu:\ \mu\in \R_{\geq 0}^d \ , \ \tr(\mu)\geq t-t_0\ \}\ .
$$
In this context (see \cite{MRS4}) we showed that there exists a unique 
$\nuel\in \Lambda_t(\la)$ such that 
$$
\nuel=\nuel ^\downarrow \ \text{ and }\ \nuel\prec_w\ \nu \ \text{ for every } \ \nu\in \Lambda_t(\la)\, , 
$$ 
i.e., $\nuel$ is a $\prec_w$-minimizer in $\Lambda_t(\la)$.
Moreover, $\nuel$ can be explicitly computed as follows: let 
$h:[\lambda_d \coma \infty)\rightarrow \R_{\geq 0}$ be given by $h(x)=\sum_{i\in\IN{d}}(x-\lambda_i)^+$, where $\alpha^+$ stands for the positive part of $\alpha\in\R$.
Then, $h$ is a continuous and strictly increasing function in its domain, and for $t\geq t_0$ there exists a unique $c_\la(t)\in[\la_d \coma \infty)$ such that $h(c_\la(t))=t-t_0\,$. Then, with these notations we have that $ \nuel=(\la_i+(c_\la(t)-\la_i)^+)_{i\in\IN{d}}\in \R^d$ i.e. 
$$
\nuel=c_\la(t)\cdot \uno_d \ \ \ (\text{ if } \ c_\la(t)> \la_1) \ \text{ or } 
\ \nuel=(\la_1 \coma \ldots\coma \la_r\coma c_\la(t)\cdot \uno_{d-r}) 
\ \ \ (\text{ if } 
\ c_\la(t)\leq \la_1) 
$$
for some $r\in\IN{d}\,$. Notice that $\nuel=\nuel^\downarrow$ and $\tr(\nuel)=t$ in any case. Also 
\beq \label{sobre el orden1}
 \nuel - \la= \big(\,(c_\la(t)-\la_i)^+\big)_{i\in\IN{d}} \ \implies  \ 
 \nuel - \la=(\nuel - \la)\ua\ .
 \eeq
 Moreover, $\nuel$ is determined as the unique $\nu\in\Lambda_t(\la)$ such that $\nu$ is a $\prec_w$ minimizer in $\Lambda_t(\la)$ and such that $\nu-\la=(\nu-\la)\ua$.
\EOE

\pausa
The following lemma is a direct consequence of the results from \cite{MRS4} described in Remark \ref{resultados previos1} above. 

\end{rem}
\begin{lem}\label{nu unico}\rm
Let $\lambda=(\la_i)_{i\in\IN{d}}\in(\R^d)\da\,$, and 
$t>t_0\igdef \tr \, \la $.  Assume that $r\in \IN{d-1}$ and $c>0$ are such that the vector  
$\gamma \igdef (\la_1\coma \dots \coma \la_r\coma c\, \uno_{d-r}) $ satisfies that 
\beq\label{el nu}
\la_{r}\ge c\ge \la_{r+1}   \peso{(so that \ $\gamma =\gamma\da  \geqp \la$) \quad and} 
\tr \, \gamma = t\ .
\eeq Then, in this case we have that $c= c_\la(t)$ and $\ga = \nuel$.
\qed
\end{lem}

\normalsize
\pausa
The following statement  finds  a minimum for submajorization in the set 
$\Lambda_{(t \coma \eps)}(\la  \coma 0)$. Note that  $\Lambda_{(t \coma \eps)}(\la  \coma m)
=  \Lambda_{(t \coma \eps)}(\la  \coma 0)$ for every $m\le0$.

\begin{teo}\label{mayo min mayo 2} 
Let $\la=(\la_i)_{i\in\IN{d}}\in(\R^d)\da$ and $t_0=\tr \,\la $. Let $\eps>0$ and 
$t\geq t_0$ be such that $t-t_0\leq d\cdot \eps$. 
Then there exists $\rho\in \Lambda_{(t \coma \eps)}(\la  \coma 0)$, 
such that 
$$
\tr(\rho)=t  \py \rho\prec_w \gamma  \peso{for every} \gamma\in \Lambda_{(t \coma \eps)}(\la  \coma 0) \ . 
$$
In this case we can choose a unique $\rho=\rho_{(t \coma \eps)}(\la,0)$ as above and such that $\rho-\la=(\rho-\la)\ua$. 
Moreover, this vector $\rho$ also satisfies that $\rho=\rho\da\in(\R^d)\da$. 
\end{teo}
\proof Assume first that $\la\in(\R_{\geq 0}^d)\da$. In what follows we consider the notations from Remark \ref{resultados previos1}. We shall construct the vector $\rho_{(t \coma \eps)}(\la,0)=\rho=(\rho_i)_{i\in\IN{d}}$ 
recursively as follows: 
\ben
\item If $c_\la(t)-\la_d \le \eps$ then just take $\rho = \nuel$. 
\item If $c_\la(t)-\la_d > \eps$ then we put $\rho_d = \la_d +\eps$ and we consider the following new data: 
$$
\la^{(d-1)} = (\la_1\coma \dots \coma \la_{d-1}) \py t^{(d-1)} = t -\rho_d \ \ (\implies 
t^{(d-1)} - \tr \, \la^{(d-1)}\le (d-1)\, \eps\  )\ .
$$
Then go back to the first step, but applied to the pair $(\la^{(d-1)}\coma t^{(d-1)}\,)$. \EOE
\een
The hypothesis that $t-t_0\leq d\cdot \eps$ assures that this process stops (at 
some step $d-m+1\in \IN{d}$) obtaining the outcome $\rho=(\nu(\la^{(m)}\coma t^{(m)}),\la_{m+1}+\eps,\ldots,\la_d+\eps)\in \R_{\geq 0}^d$. That is, 
\beq\label{el rho}
\rho = (\la_1 \coma \dots \la_s \coma c\, \uno_{m-s} \coma \la_{m+1}+\eps \coma \dots \coma \la_d+\eps) \peso{with}
 0\le c-\la_m \le \eps 
\eeq
and $\la_{s+1} \le c<\la_s $ 
in case $(\la_1 \coma \dots \la_s \coma c\, \uno_{m-s}\, )  = \nu(\la^{(m)}\coma t^{(m)})$ ,  or 
\beq\label{el rho2c}
\rho = ( c\, \uno_{m} \coma \la_{m+1}+\eps \coma \dots \coma \la_d+\eps) 
\peso{with}
 0\le c-\la_m \le \eps  \, ,
\eeq in case $\nu(\la^{(m)}\coma t^{(m)}) = c\, \uno_m\,$.
It is clear that this $\rho \in \Lambda_{(t \coma \eps)}(\la  \coma 0) $ and that $\tr \, \rho = t$. 
Moreover, we claim that $\rho = \rho\da$: notice that we only need to show that $c \ge \la_{m+1}+\eps$, where $c=c_{\la^{(m)}} (t^{(m)})$.  

\pausa
Indeed, since the algorithm did not stop at the pair $(\la^{(m+1)},t^{(m+1)})$
then $c_{\la^{(m+1)}} (t^{(m+1)}) > \la_{m+1}+\eps$. 
By Remark \ref{resultados previos1} and Lemma \ref{nu unico}, it is easy to see that 
$$ c_{\la^{(m+1)}} (t^{(m+1)})=c_{\la^{(m)}} (t^{(m+1)}- c_{\la^{(m+1)}} (t^{(m+1)}))\ .$$
Moreover, since $c_{\la^{(m)}} (x)$ is an increasing function then
$$ c_{\la^{(m)}} (t^{(m+1)}- c_{\la^{(m+1)}} (t^{(m+1)}))\le  c_{\la^{(m)}} \big(\, t^{(m+1)}- (\la_{m+1}+\eps)\, \big) =   
c_{\la^{(m)}} (t^{(m)}) = c \ . 
$$
Hence, $ \la_{m+1}+\eps \le c_{\la^{(m+1)}} (t^{(m+1)})  \le c$, which shows that $\rho = \rho\da$.  

\pausa
Fix $\ga \in \Lambda_{(t \coma \eps)}(\la  \coma 0) $ such that $\gamma=\la+\mu$, where $\mu=(\mu_i)_{i\in\IN{d}}$ is such that $0\leq \mu_i\leq \eps$ for $i\in \IN{d}$ and $\tr(\mu)\geq t-\tr(\la)$. Then, with the notation of Remark 
\ref{resultados previos1}, the truncation 
$(\ga_1 \coma \dots \coma \ga_m) \in 
\Lambda_{t^{(m)}}(\la^{(m)} )$ because $\la_i\leq \ga_i$ for $i\in \IN{m}$ and 
$$
\suml _{i=m+1}^d \ga_i -\la_i \le (d-m)\, \eps = \suml _{i=m+1}^d \rho_i -\la_i
\implies 
\tr \, (\ga_1 \coma \dots \coma \ga_m) \ge \tr \,(\rho_1 \coma \dots \coma \rho_m)=  t^{(m)} \ . 
$$
By construction  $(\rho_1 \coma \dots \coma \rho_m)  = \nu(\la^{(m)}\coma t^{(m)})
  \prec_w (\ga_1 \coma \dots \coma \ga_m)$. 
Hence, we get 
\beq \label{desig mayo3}
\sum_{i=1}^k \rho_i\leq \sum_{i=1}^k \gamma_i \peso{for every} 
 k\in\IN{m}\ .
\eeq
On the other hand, if $m+1\leq k\leq d$ then 
\beq \label{desig mayo4}
\sum_{i=1}^k \rho_i=t-\sum_{i=k+1}^d \rho_i=t-\sum_{i=k+1}^d (\la_i+\eps)\leq \tr(\gamma)-\sum_{i=k+1}^d (\la_i+\mu_i)
=\sum_{i=1}^k \gamma_i \ ,
\eeq since $0\leq \mu_i\leq \eps$ for $i\in\IN{d}\,$. Thus, Eqs. \eqref{desig mayo3} and \eqref{desig mayo4} imply that $\rho\prec_w\gamma$, since $\rho=\rho\da$ (even when the entries of $\gamma$ are not necessarily arranged in non-increasing order). 
The fact that 
$$\rho-\la = (\nu(\la^{(m)}\coma t^{(m)}) -\la^{(m)}\coma \eps\cdot \uno_{d-m}) =
(\rho-\la)\ua$$ 
follows from Eq. \eqref{sobre el orden1} and that $c-\la_m \le \eps$. 

\pausa
Now, assume that  $\la_d < 0$.  
Take some  $s > -\la_d\,$, so that the translated vector 
$\la+s\, \cdot \, \uno_d \in (\R_{\geq 0}^d)\da$. Let 
$\rho=\rho_{(t+d\,\cdot\, s\coma  \eps)}(\la+s\, \cdot \, \uno_d\coma 0)$ 
be the $\prec_w$-minimizer in 
$\Lambda_{(t+d\,\cdot \,s\coma \eps)}(\la+s\cdot \uno_d \coma 0)$ 
constructed in the first part of this proof. Since 
$\Lambda_{(t+d\,\cdot \,s\coma \eps)}(\la+s\cdot \uno_d \coma 0)
=\Lambda_{(t\coma \eps)}(\la  \coma 0)+s\cdot \uno_d$ we see that $\rho-s\cdot \uno_d\in \Lambda_{(t\coma \eps)}(\la  \coma 0)$ is a $\prec_w$-minimizer in $\Lambda_{(t\coma \eps)}(\la  \coma 0)$.

\pausa
Finally, assume that $\rho'\in \Lambda_{(t\coma \eps)}(\la  \coma 0)$ is a $\prec_w$-minimizer in 
$\Lambda_{(t\coma \eps)}(\la  \coma 0)$ and such that $\rho'-\la=(\rho'-\la)\ua$. 
If $\rho=\rho_{(t \coma \eps)}(\la,0)$ is as before, then: in case $\rho_d=\la_d+\eps$ 
(respectively $\rho'_d=\la_d+\eps$) then it is easy to see that also $\rho'_d=\la_d+\eps$ 
(respectively $\rho_d=\la_d+\eps$); this observation allows reduce the problem of uniqueness 
of $\rho$ to the case in which $\rho_d<\eps$ and $\rho'_d< \eps$, where uniqueness of $\rho$ 
follows from the comments after Eq. \eqref{sobre el orden1} in Remark \ref{resultados previos1}.
\QED

\begin{rem}\label{se calcula efec 0}
With the notations of Theorem \ref{mayo min mayo 2}, notice that the proof 
of that result shows an explicit and simple algorithm that computes the 
optimal vector $\rho$ in terms of the vector $\nu(\tilde \la\coma  \tilde t)$
described in Remark \ref{resultados previos1}, for an explicit 
(and computable) $\tilde \la\in \R^r$ and $\tilde t\in \R$. Since 
$\nu(\tilde \la\coma  \tilde t)$ can be also computed in terms of a 
simple algorithm (as described in Remark \ref{resultados previos1}), we see 
that the optimal vector $\rho$ can be effectively computed (with a 
fast algorithm).
\EOE 
\end{rem}

\pausa
The following result complements Theorem \ref{mayo min mayo 2}.
\begin{teo}\label{cor mayo min mayo 2}
Let $\la\in(\R^d)\da$ and $t_0=\tr \,\la $. Let $\eps>0$, $m\in\mathbb Z$ with $m\leq d-1$ and 
let $t\geq t_0$ be such that $t-t_0\leq \min\{(d-m),\,d\}\cdot \eps$. 
Then there exists $\rho\in \Lambda_{(t \coma \eps)}(\la  \coma m)$, 
such that 
$$
\tr(\rho)=t  \py \rho\prec_w \gamma  \peso{for every} \gamma\in \Lambda_{(t \coma \eps)}(\la  \coma m) \ . 
$$
Moreover, we can choose a unique $\rho=\rho_{(t \coma \eps)}(\la,m)$ as above such that $\rho-\la=(\rho-\la)\ua$ (although the entries of $\rho$ may not be arranged in non-increasing order).
\end{teo}
\proof
If we assume that $m\leq 0$ then we let 
$\rho_{(t \coma \eps)}(\la,m)\igdef\rho_{(t \coma \eps)}(\la\coma 0)$. 
Since in this case
$\Lambda_{(t \coma \eps)}(\la  \coma m)=\Lambda_{(t \coma \eps)}(\la  \coma 0)$, Theorem \ref{mayo min mayo 2} implies that  
$\rho_{(t \coma \eps)}(\la\coma m)$ has the desired properties.

\pausa
Assume now that $m\in\IN{d-1}$. Set 
$\tilde \la=(\la_{m+1}\coma \ldots\coma \la_d)\in(\R^{d-m})\da$, 
$s = t- \sum_{i-1}^m \, \la_i \, $, 
and consider $\rho_{(s\coma\eps)}(\tilde \la\coma 0)\in 
\Lambda_{(s\coma \eps)}(\tilde\la  \coma 0)\inc \R^{d-m}$ as in Theorem \ref{mayo min mayo 2}. Then we define 
\beq\label{elrho}
\rho = \rho_{(t\coma \eps)}( \la\coma m)\igdef
(\la_1\coma \ldots\coma \la_m \coma \rho_{(s\coma \eps)}(\tilde \la\coma 0))
\in \R^d \ .
\eeq
Clearly, $\rho
\in \Lambda_{(t\coma \eps)}(\la  \coma m)$ and $\tr ( \rho) = t$.
Let $\gamma=\la +\mu\in \Lambda_{(t\coma \eps)}(\la  \coma m)$, so that 
$0\leq \mu_i\leq \eps$ for $i\in \IN{d}\,$, supp$(\mu)\leq d-m$ and 
$\tr(\mu)\geq t-t_0\,$.
By Lidskii's (additive) inequality we get that $\la+\mu\ua\prec\gamma$. 

\pausa
But $\mu\ua=(0\cdot \uno_m \coma \tilde \mu)$, where 
$\tilde \mu=(\tilde \mu_i)_{i\in\IN{d-m}}\in (\R^{d-m})\ua$ is such 
that $0\leq \tilde \mu_i\leq \eps$ for $i\in\IN{d-m}$ and 
$\tr(\tilde\mu)=\tr(\mu)\geq t-t_0\, = s - \tr (\tilde \la )$. 
Hence, $\tilde \la+\tilde\mu\in 
\Lambda_{(s \coma \eps)}(\tilde \la  \coma 0)$ and therefore 
$\rho_{(s \coma \eps)}(\tilde \la \coma 0)\prec_w\tilde \la+\tilde\mu$, 
by Theorem \ref{mayo min mayo 2}. Thus, using the fact that we 
have submajorization by blocks as in Remark \ref{rem doublystohasctic}, 
$$
\rho_{(t \coma \eps)}(\la \coma m)
=(\la_1\coma \ldots \coma \la_m\coma \rho_{(s \coma \eps)}(\tilde \la \coma 0)) 
\prec_w (\la_1\coma \ldots \coma \la_m\coma \tilde \la 
+\tilde \mu)=\la+\mu\ua\prec \gamma\ .
$$
Finally, notice that the vector $\rho$ as defined in Eq. 
\eqref{elrho}  
may be not arranged in non-increasing order. 
Nevertheless, according to Theorem \ref{mayo min mayo 2} 
$$
\rho_{(t \coma \eps)}(\la  \coma m)-\la=(0\cdot \uno_{m}\coma 
\rho_{(s \coma \eps)}(\tilde \la \coma 0)-\tilde \la)\in (\R^d)\ua
\ . 
$$
Let $\rho'\in\Lambda_{(t \coma \eps)}(\la  \coma m)$ be a $\prec_w$-minimizer 
in $\Lambda_{(t \coma \eps)}(\la  \coma m)$ and such that $\rho'-\la=(\rho'-\la)\ua$. Then it is easy to see that 
$\rho'=(\la_1\coma \ldots \coma \la_m\coma \rho'')$, where 
$\rho''\in \Lambda_{(s \coma \eps)}(\tilde \la  \coma 0)$ 
is a $\prec_w$ minimizer in $\Lambda_{(s \coma \eps)}(\tilde \la  \coma 0)$ 
and such that $\rho''-\tilde\la=(\rho''-\tilde\la)\ua$. Hence, by 
Theorem \ref{mayo min mayo 2}, we conclude that 
$\rho''=\rho_{(s \coma \eps)}(\tilde \la \coma 0)$ and therefore 
$\rho'=\rho_{(t \coma \eps)}(\la \coma m)$.
\QED

\pausa
Based on Lidskii's inequality, Theorem \ref{cor mayo min mayo 2} allows to compute the structure of optimal elements in 
$U_{(t \coma \eps)}(S,m)$ from its rather naive model in $\R^d$. 

\begin{teo} \label{teo para ops}
Let $S\in\matpos$, $t>t_0\igdef\tr(S)$, $\eps>0$ and $m\in\mathbb Z$ with $m\leq d-1$ such that $t-t_0\leq \min\{(d-m),\,d\}\cdot \eps$. 
Let $\la=\la(S)$ and $\rho=\rho_{(t \coma \eps)}(\la \coma m)\in 
\Lambda_{(t \coma \eps)}(\la\coma m)$ be the $\prec_w$ minimizer from Theorem \ref{cor mayo min mayo 2}. Then,
\ben
\item There exists $S_0\in U_{(t \coma \eps)}(S,m)$ such that $\la(S_0)=\rho\da$; 
\item $S_1\in U_{(t \coma \eps)}(S,m)$ is such that $\la(S_1)\prec_w \la(S')$ for every $S'\in U_{(t \coma \eps)}(S,m)  \iff 
\la(S_1)=\rho\da$. 
\item If $S+B\in U_{(t \coma \eps)}(S,m)$ is such that $\la(S+B)=\rho\da$ then
there exists an o.n.b. $\{v_j\}_{j\in\IN{d}}$ for $\C^d$ 
such that, with the notations of \eqref{tensores}, 
$$
S=\sum_{j\in\IN{d}} \lambda_j(S)\ v_j\otimes v_j \peso{and} B= \sum_{j\in\IN{d}} \lambda_{d-j+1}(B) \ v_j\otimes v_j \ .
$$
\een
\end{teo}
\begin{proof}
Let $\{w_j\}_{j\in\IN{d}}$ be an o.n.b. for $\C^d$ such that 
$S=\sum_{j\in\IN{d}} \lambda_j\ w_j\otimes w_j$, where $\la(S)=\la=(\la_i)_{i\in\IN{d}}\,$.
Let $\rho=\rho_{(t \coma \eps)}(\la\coma m)$ be as in Theorem \ref{cor mayo min mayo 2} and let $\mu=\rho-\la$, so that $\mu=\mu\ua$. 
Let $B_0=\sum_{j\in\IN{d}} \mu_j\ w_j\otimes w_j$ and notice that then $B_0\in\matpos$, $\|B_0\|\leq \eps$, $\tr(B_0)=t-\tr(S)$ and $\rk(B_0)=$supp$(\mu)\leq m-d$, so that $S+B_0\in U_{(t \coma \eps)}(S,m)$. Moreover, by construction $\lambda(S+B_0)=\rho\da$. 

\pausa 
Let $S'\in U_{(t\coma \eps)}(S\coma m)$ and let $B\in\matpos$, $\|B\|\leq \eps$, $\rk(B)\leq d-m$ and $\tr(B)\geq t-\tr(S)$ so that 
$S'=S+B$. By Lidskii's additive inequality we conclude that $\lambda+\lambda(B)\ua\prec \lambda(S')$.
Now, by the hypothesis on $B$ we conclude that $\lambda+\lambda(B)\ua\in\Lambda_{(t \coma \eps)}(\la\coma m)$. Hence 
$$ \rho\prec_w \lambda(S)+\lambda(B)\ua\prec \lambda(S')\ . $$ 
Hence, if $S_1\in U_{(t\coma \eps)}(S\coma m)$ is such that $\la(S_1)=\rho\da$ then $\la(S_1)\prec_w\la(S')$ for every $S'\in
U_{(t\coma \eps)}(S\coma m)$. Conversely, if  $S_1\in U_{(t\coma \eps)}(S\coma m)$ is such that 
$\la(S_1)\prec_w\la(S')$ for every $S'\in
U_{(t\coma \eps)}(S\coma m)$ then $\la(S_1)\prec_w\la(S_0)\prec_w\la(S_1)$ implies that $\la(S_1)=\la(S_0)=\rho\da$.

\pausa 
Finally, if $S+B\in U_{(t \coma \eps)}(S,m)$ is such that $\la(S+B)=\rho\da$ then
by Lidskii's additive inequality, the fact that $\lambda(S)+\lambda(B)\ua\in\Lambda_{(t \coma \eps)}(\la\coma m) $  and Theorem \ref{cor mayo min mayo 2} we see that $$
\lambda+\lambda(B)\ua\prec_w\lambda(S+B)=\rho\da\prec_w \lambda
+\lambda(B)\ua \implies \la(S+B)= (\lambda+\lambda(B)\ua )\da\ .
$$ 
Hence, the existence of the o.n.b. $\{v_j\}_{j\in\IN{d}}$ 
of item 3 
is a consequence of \cite[Theorem 8.8]{MRS12} (the case of equality in Lidskii's additive inequality). 
\end{proof}

\begin{rem}\label{se calcula efectivo}
The proof of Theorem \ref{cor mayo min mayo 2} together with Remark \ref{se calcula efec 0} show that 
the vector $\rho=\rho_{(t \coma \eps)}(\la,m)\in \Lambda_{(t \coma \eps)}(\la\coma m)$ as in the statement of Theorem \ref{teo para ops} can be explicitly computed in terms of a simple (and fast) algorithm depending on $t$, $\eps$, $\la$ and $m$.
\EOE \end{rem}

\pausa

\subsection{Computation of optimal perturbations of the canonical dual frame}\label{aplic dual}

\begin{nota}\label{let F}
Let $\F=\{f_j\}_{j\in\In}$ be a frame for $\C^d$ and set $m=2d-n$. Let 
$\eps>0$ and $t> \tr(S_{\F^\#})$  be such that $t-\tr(S_{\F^\#})\leq \min\{(d-m),\,d\}\cdot \eps^2$. Recall that Proposition \ref{carac de ops frame duals} and Eq. \eqref{defi ut} imply that 
\beq\label{la =}
\cS(\cD_{(t \coma \eps)}(\cF))=U_{(t \coma \eps^2)}(S_{\cF^\#}\coma m)\ . 
\eeq
\end{nota}
\begin{teo}\label{duals opt}
Consider the Notations 
\ref{let F}.  
Let $\la=\la(S_{\cF^\#})$ and $\rho=\rho_{(t\coma \eps^2)}(\la\coma m) $ be the $\prec_w$ minimizer 
for $ \Lambda_{(t \coma \eps^2)}(\la\coma m)$ given in  Theorem \ref{cor mayo min mayo 2}. Then,
\ben
\item  There exists $\cG_0\in \cD_{(t \coma \eps)}(\cF)$ such that $\la(S_{\cG_0})=\rho\da$.

\item Fix any $\cG\in \cD_{(t \coma \eps)}(\cF)$. Then this $\cG$ satisfies that 
$$S_{\cG}\prec_w S_{\cG'} \peso{for every} \cG'\in\cD_{(t \coma \eps)}(\cF) \iff \la(S_{\cG})=\rho\da
\ .
$$
\item Given $\G \in \cD_{(t \coma \eps)}(\cF)$,   then $\la(S_{\G})=\rho\da 
\iff \cG=\{f_i^\#+k_i\}_{i\in\In}$ for some $\cK=\{k_i\}_{i\in\In}$ such that $T_\cF^*\,T_\cK=0$ 
and such there exists an o.n.b. $\{v_j\}_{j\in\IN{d}}$ for $\C^d$ with
\beq \label{eq sk}
S_{\F^\#}=\sum_{j\in\IN{d}} \lambda_j\ v_j\otimes v_j \peso{and} S_{\cK}= \sum_{j\in\IN{d}} (\rho-\la)_j\ v_j\otimes v_j \ .
\eeq
In this case $S_\cG=S_{\cF^\#}+S_{\cK}\,$; in particular, $S_{\F^\#}$ and $S_{\G}$ commute.
\een
\end{teo}
\begin{proof}
It is a direct consequence of Eq. \eqref{la =},  Theorem \ref{teo para ops} and the characterization of dual frames given in 
Proposition \ref{carac de ops frame duals}. 
\end{proof}

\pausa
Notice that Theorem \ref{duals opt} contains a procedure to compute optimal duals $\G=\{g_i\}_{i\in\In}\in \cD_{(t \coma \eps)}(\cF)$. In what follows, given $h\in \convf$ we consider its associated convex potential $P_h$ on finite sequences in $\C^d$ given by $P_h(\cF)=\tr(h(S_\cF))=\sum_{i=1}^dh(\la_i(S_\cF))$.

\begin{cor}\label{cor energy}
Fix an increasing function $h\in \convf$. With the notations and terminology of Theorem \ref{duals opt}, 
the following inequality holds: 
\beq \label{lowb}
P_h(\cG)\geq \sum_{i\in\IN{d}} h(\rho_i) \peso{for every} \cG\in \cD_{(t \coma \eps)}(\cF) \ , 
\eeq and this lower bound is attained.
If we assume further that $h\in \convfs$,  then $\cG\in \cD_{(t \coma \eps)}(\cF)$ attains the lower bound in
\eqref{lowb}  $\iff \cG=\{f_i^\#+k_i\}_{i\in\In}$ for some $\cK=\{k_i\}_{i\in\In}$ such that 
$T_\cF^*T_\cK=0$ and such there exists an o.n.b. $\{v_j\}_{j\in\IN{d}}$ for which Eq. \eqref{eq sk} holds.
\end{cor}
\begin{proof}
It follows from Theorems \ref{duals opt} and the standard results of Section \ref{subsec prelims mayo}. 
\end{proof}

\begin{rem}
There are some $h\in\convf$ for which their associated convex potential $P_h$ can be computed in a rather direct way (i.e., without necessarily computing the eigenvalues of the frame operator of the sequence of vector). For example, if $h(x)=x^2$ then $P_h=\FP$ is the so-called frame potential. In this case, given a sequence $\cG=\{g_i\}_{i\in\In}$ then it is well known that $$P_h(\cG)=\FP(\cF)=\sum_{i,\,j\in\In}|\langle g_i\coma g_j\rangle |^2\ .$$
Now, consider a fixed frame $\F=\{f_j\}_{j\in\In}$ and assume that $\la(S_\cF)$ is a known data. 
Set $m=2d-n$, let $\eps>0$ and $t> \tr(S_{\F^\#})$ be such that $t-\tr(S_{\F^\#})\leq \min\{(d-m),\,d\}\cdot \eps^2$.
In this case $\la(S_{\cF^\#})$ is also a known data and therefore $\rho\in\R^d$ as in 
Corollary \ref{cor energy} can be explicitly computed (see Remark \ref{se calcula efectivo}). Thus, according to Corollary \ref{cor energy}  above we get 
\beq\label{ineq part}
\FP(\cG)=\sum_{i\coma j\in\In}|\langle g_i\coma g_j\rangle|^2\geq \sum_{i\in\IN{d}}\rho_i^2 \ \coma \ \text{for every}\ \cG\in  
\cD_{(t \coma \eps)}(\cF)\ .
\eeq
The previous inequality provides a quantitative criteria for checking the optimality of $\cG$. That is, the closer 
$\FP(\cG)$ is to this explicit lower bound, the more concentrated $\la(S_\cG)$ is (which is the type of analysis that originally motivated the introduction of the frame potential). Indeed, since $h(x)=x^2$ is strictly convex, Corollary \ref{cor energy} and Theorem \ref{duals opt} imply that if $\cG\in  \cD_{(t \coma \eps)}(\cF)$ attains the lower bound in Eq. \eqref{ineq part} then $\la(S_{\cG})$ has minimal spread (in the sense that is a $\prec_w$-minimizer) among $\la(S_{\cG'})$ for $\cG'\in \cD_{(t \coma \eps)}(\cF)$. Moreover, in this case the geometrical structure of $S_\cG$ can be described explicitly in terms of the geometrical structure of $S_\cF$.
\EOE
\end{rem}

\section{Optimal perturbations by equivalent frames}\label{sec 5 opt cohe pert}

In this section, given a frame $\cF$, we consider a matrix model for the design of perturbations of the identity $V$ such that $V\cdot \cF$ has the desired optimality properties. In Section \ref{model equiv}, using tools from matrix analysis, we show that there are optimal solutions within our abstract model. Along the way, we will prove some results that are interesting in their own right, and develop some aspects of the multiplicative Lidskii's inequality with respect to log-majorization (see Section \ref{sec append}). Then, in Section \ref{aplic equiv} we apply these results to the initial frame design problem.

\subsection{ A matrix model for $\cS(\cP_{(s \coma \delta)}(\cF))$}\label{model equiv}

In order tackle the problem of the existence and computation of optimal perturbations by equivalent frames of a fixed frame
we introduce the following matrix model: given $S\in \matpos$, $0<\delta<1$ and $s\in[(1-\delta)^d,(1+\delta)^d]$ then consider
\beq \label{defi cos}
\cO_{(s \coma \delta)}(S)=\{V^*S V:\ V\in\cG l(\hil)\ , \ \|V^*V-I\|\leq \delta \ , \ \det(V^*V)\geq s\ \}\ . 
\eeq
With the notation of Section \ref{sec opt cohe pert} and Eqs.  \eqref{defi cp}, \eqref{defi scp}, 
since $S_{V\cdot\cF}= VS_\cF V^*$ then Eq. \eqref{ident scp1} becomes $\cS(\cP_{(s \coma \delta)}(\cF))=\cO_{(s \coma \delta)}(S_\cF)$.

\pausa
The following result, that is the multiplicative analogue of Lidskii's (additive) inequality, will play a key role in our work on frames. We develop its proof in an Appendix (see Section \ref{sec append}).

\begin{teo}\label{teo hay max y min mayo}\rm
Let $S\in\matinvd^+$ and let $\ga\in (\R_{>0}^d)\da$. 
Then, for every $V\in\matinvd$ such that $\lambda(V^*V)=\ga$ we have that 
\beq\label{max y min logamyo}
\lambda(S)\circ \ga\ua\prec_{\log}\la(V^*SV)\prec_{\log}\lambda(S)\circ \ga
\in (\R_{>0}^d)\da \ .
\eeq 
Moreover, if $\lambda(V^*SV)=(\lambda(S)\circ \ga\ua)\da$ 
(resp. 
$\lambda(V^*SV)=\lambda(S)\circ \ga$) then there exists an o.n.b. 
$\{ v_i\}_{i\in \IN{d}}$ of $\C^d$ such that 
\beq\label{hay base}
S=\sum_{i\in \IN{d}} \la_i(S)\ v_i\otimes v_i \peso{and} |V^*|=
\sum_{i\in \IN{d}} \ga_{d+1-i}\rai\ v_i\otimes v_i\  
\eeq 
\big(\, resp. $S=\sum_{i\in \IN{d}} \lambda_i(S)\ v_i\otimes v_i$ and 
$|V^*|=\sum_{i\in \IN{d}} \ga_i\rai 
\ v_i\otimes v_i$\,\big).
\qed
\end{teo}

\pausa
The previous result allows to show the existence of optimal (minimal spread) elements in $\cO_{(s \coma \delta)}(S)$ (see Eq. \eqref{defi cos}); we further compute their  geometric structure.
\begin{teo}\label{teo perturb multi}
Let $S\in\matinvd^+$, $0<\delta<1$ and let $s \in \big[\,(1-\delta)^d\coma (1+\delta)^d\, \big]$. 
Define the following data: $\la=\log \, \la(S) = \big(\,\log \,\lambda_i(S)\,\big)_{i\in\IN{d}}\in(\R^d)\da\ $,  
$$t=\log\left( \frac{s\cdot \det(S)}{(1-\delta)^d}\right)\geq \tr(\la) \peso{and} \eps
=\log\left( \frac{1+\delta}{1-\delta}\right)>0\ .
$$ 
Let $\rho=\rho_{(t \coma \eps)}(\la  \coma 0)\in\Lambda_{(t \coma \eps)}(\la  \coma 0)$ be as in Theorem \ref{mayo min mayo 2}. 
Denote by $\mu=\mu_{(s\coma \delta)}(S)$ the vector
\beq \label{el muel}
\mu = (1-\delta)\cdot \exp \, \rho = \big(\, (1-\delta) \, e^{\,\rho_i} 
\,\big)_{i\in\IN{d}}\in(\R_{>0}^d)\da\ .
\eeq 
Then, 
\ben
\item There exists $\tilde S_0\in \cO_{(s\coma \delta)}(S)$ such that $\la(\tilde S_0)=\mu$.
\item For every $\tilde S\in \cO_{(s\coma \delta)}(S)$
we have that 
\beq \label{desi teo}
\prod_{i=1}^k \mu_i 
\leq \prod_{i=1}^k\lambda_i(\tilde S)  \peso{for every} k\in \IN{d}  \  .
\eeq 
\item 
Given $\tilde S=V^*SV\in \cO_{(s\coma \delta)}(S)$
then equality holds in Eq. \eqref{desi teo} for every $k\in \IN{d}$ (i.e. $\mu=\la(\tilde S)$) $\iff$ 
$\det(V^*V)=s$ and there exists an o.n.b. $\{v_i\}_{i\in \IN{d}}$ for $\C^d$ such that 
\beq\label{hay base 2}
S=\sum_{i\in\IN{d}}\lambda_i(S)\ v_i\otimes v_i \peso{and} 
VV^* 
=\sum_{i\in\IN{d}} \frac{\mu_i}{\la_i(S)}\ v_i\otimes v_i\ . 
\eeq
\een
\end{teo}
\begin{proof}
First notice that $$t-\tr(\la)=\log\left(\frac{s}{(1-\delta)^d}\right)\leq d\cdot \log\left( \frac{1+\delta}{1-\delta}\right)=d\cdot \eps\ .$$ Hence, we can compute $\rho_{(t \coma \eps)}(\la  \coma 0)$ as in Theorem \ref{mayo min mayo 2}. 

\pausa
Let $\{v_i\}_{i\in\IN{d}}$ be an o.n.b. for $\C^d$ such that $Sv_i=\lambda_i(S)\,v_i$ for $i\in\IN{d}\,$. Then we define 
$V=\sum_{i\in\IN{d}}(\frac{\mu_i}{\la_i(S)})^{1/2}\ v_i\otimes v_i\in\matinvd^+$ and $\tilde S_0=VSV$. It is clear that $\la(\tilde S_0)=\mu$. 
Since $\log \la(V^2)= \log(1-\delta)\cdot\uno_d+\rho-\la $, we get that 
$$\log(1-\delta)\leq \log \la_i(V^2) \leq \log(1-\delta)\, \eps=\log (1+\delta) \implies  1-\delta\leq \la_i(V^2)\leq 1+\delta \coma \quad i\in\IN{d}\ , $$ which is equivalent to $\|V^2-I\|\leq \delta$. Also, $\det(V^2)=(1-\delta)^d\, \det(S)^{-1}\, \exp(\tr(\rho))=s$ so that $\tilde S_0\in \cO_{(s\coma \delta)}(S)$ with $\la (\tilde S_0)=\mu$.

\pausa
In order to show items 2 and 3, let us first assume that $V\in \matinvd^+$. 
Then, by Theorem \ref{teo hay max y min mayo},  
$$\la(S)\circ \la(V^2)\ua\prec_{\log} \la(VSV) \py 
(\la(S)\circ \la(V^2)\ua)\da=\la(VSV)$$ 
if and only if there exists an o.n.b. 
such that Eq. \eqref{hay base} holds. 
Therefore
\beq\label{reduc lidskii}
\log(\la(S))+\log(\la(V^2))\ua=\log(\la(S)\circ \la(V^2)\ua) \prec \log(\la(VSV))\ .
\eeq 
Assume further that $\det(V^2)\geq s$ and $\|I-V^2\|\leq \delta$. Then $1-\delta\leq \la_i(V^2)\leq 1+\delta$ and hence
$\log(1-\delta)\leq \log(\la_i(V^2))\leq \log(1+\delta)$ for every  $i\in\IN{d}\,$.
On the other hand 
$$ 
\tr\,\log(\la(V^2))=\log(\det \, V^2)\geq \log(s) 
\ .
$$ 
Hence $\gamma=\log(\la(V^2))\ua- \log(1-\delta)\cdot \uno\in (\R_{\geq 0}^d)\da$ is such that $$\tr(\gamma)\geq \log\left(\frac{s}{(1-\delta)^d}\right) \peso{and} 0\leq \gamma_i\leq \log(1+\delta)-\log(1-\delta)=\log\left(\frac{1+\delta}{1-\delta}\right) \ , \ i\in\IN{d}\ .$$
Thus, using the notations in the statement we see that $$\log(\la(S))+\gamma=\la+\gamma\in\Lambda_{(t \coma \eps)}(\la  \coma 0) \implies \rho\prec \la+\gamma\ . $$
Therefore, in this case we have that 
$\rho+\log(1-\delta)\cdot \uno\prec_w \la+\log(\la(V^2))\ua$. These facts together with Eq. \eqref{reduc lidskii} imply that 
$$ \rho+\log(1-\delta)\cdot \uno\prec_w  \log(\la(S))+\log(\la(V^2))\ua\prec \log(\la(VSV))\ . $$ Since the exponential function is convex and increasing, the submajorization relation above 
implies Eq. \eqref{desi teo}. Moreover, if equality holds in Eq. \eqref{desi teo} for $k\in \IN{d}$ then, by the previous majorization relations, we conclude that $(1-\delta)\exp(\rho)=(\la(S)\circ \la(V^2)\ua)\da=\la(VSV)$. Hence, by Theorem \ref{teo hay max y min mayo}, we see that there exists an o.n.b. $\{v_i\}_{i\in\IN{d}}$ that satisfies Eq. \eqref{hay base 2}. Conversely, notice that if there exists an o.n.b. $\{v_i\}_{i\in\IN{d}}$ for which Eq. \eqref{hay base 2} holds then it is straightforward to show that equality holds in Eq. \eqref{desi teo} for $k\in \IN{d}\,$.

\pausa
Finally, if $V\in\matinvd$ is arbitrary, the result follows from the positive case using the reduction described at the end of the proof of Proposition \ref{ostrowski}.
\end{proof}

\subsection{Computation of optimal perturbations by equivalent frames}\label{aplic equiv}

We begin with the following result, which is a consequence of Theorem \ref{teo hay max y min mayo}, and the relations between submajorization and increasing convex functions described in Section \ref{subsec prelims mayo}.

\begin{teo}\label{teo hay max y min mayo frames}
Let $\cF=\{f_i\}_{i\in\In}$ be a frame for $\C^d$, with frame operator $S_\cF\in\matinvd^+$, and fix an increasing function $h\in\convf$. Given $\ga\in (\R_{>0}^d)\da$ then,
\ben
\item If $V\in\matinvd$ is such that $\lambda(V^*V)=\ga$ then we have that 
\beq\label{max y min logamyo frames}
\sum_{i\in\IN{d}}h(\la_i(S_\cF) \ \ga_{d+1-i})\leq P_h(V\cdot \cF) \ ,
\eeq and this lower bound is attained.
\item If we assume further that $h\in\convfs$ then equality holds in Eq. \eqref{max y min logamyo frames} iff 
there exists an o.n.b. $\{ v_i\}_{i\in \IN{d}}$ such that 
\beq\label{hay base frames}S_\cF=\sum_{i\in \IN{d}} \lambda_i(S_\cF)\ v_i\otimes v_i \peso{and} |V^*|=\sum_{i\in \IN{d}} \ga_{d+1-i}\ v_i\otimes v_i\  .\eeq 
\een
\end{teo}
\begin{proof}
Recall that if $V$ and $\cF$ are as above then $S_{V\cdot \cF}=V^*S_\cF\,V$. Hence, by Theorem \ref{teo hay max y min mayo} and the remarks in Section \ref{subsec prelims mayo} we conclude that 
$$
\lambda(S_\cF)\circ \lambda(V^*V)\ua\prec_{\log} \lambda(S_{V\cdot \cF}) \implies \lambda(S_\cF)\circ \lambda(V^*V)\ua\prec_w \lambda(S_{V\cdot \cF})\ .
$$ The submajorization above together with the fact that $h$ is an increasing and convex function imply Eq. \eqref{max y min logamyo frames}. On the other hand, it is clear that this lower bound is attained.
Assume now that the lower bound in Eq. \eqref{max y min logamyo frames} is attained for some $V$ as above.
Using the fact that $h$ is (an increasing) strictly convex function and the submajorization relation above then we conclude that $(\lambda(S_\cF)\circ \lambda(V^*V)\ua)\da =\lambda(S_{V\cdot \cF}) $ (see the remarks in Section \ref{subsec prelims mayo}).
Thus, 
again by Theorem \ref{teo hay max y min mayo}, we conclude that there exists  an o.n.b. $\{ v_i\}_{i\in \IN{d}}$ for $\C^d$ for which
Eq. \eqref{hay base frames} holds. 
\end{proof}

\begin{nota}\label{con delta}
Let $\cF=\{f_i\}_{i\in\In}$ be a frame for $\C^d$ with frame operator $S_\cF\in\matinvd^+$. Let $0<\delta<1$ and let $s>0$ be such that $(1-\delta)^d\leq s\leq (1+\delta)^d$. Then, with the notations from Eqs. \eqref{defi cp}, \eqref{defi scp}
and \eqref{defi cos} then, the identity in Eq. \ref{ident scp1} becomes
$$\cS(\cP_{(s \coma \delta)}(\cF))=\cO_{(s \coma \delta)}(S_\cF)\ .$$
Thus, the following result is a direct consequence of the previous identity and Theorem \ref{teo perturb multi}.
\end{nota}

\begin{teo}\label{teo perturb multi frames}\rm
Let $\cF=\{f_i\}_{i\in\In}$ be a frame for $\C^d$ with frame operator $S_\cF\in\matinvd^+$. Let 
$\delta $ and $s$ as in \ref{con delta}.  
Define the following data: $\la=\log \, \la(S_\cF) = \big(\,\log \,\lambda_i(S_\cF)\,\big)_{i\in\IN{d}}\in(\R^d)\da\ $,  
$$t=\log\left( \frac{s\cdot \det \, S_\cF}{(1-\delta)^d}\right)\geq \tr \,\la  \peso{and} \eps
=\log\left( \frac{1+\delta}{1-\delta}\right)>0\ .
$$ 
Let $\rho=\rho_{(t \coma \eps)}(\la  \coma 0)\in\Lambda_{(t \coma \eps)}(\la  \coma 0)$ be as in Theorem \ref{mayo min mayo 2}. 
Denote by $\mu=\mu_{(s\coma \delta)}(\cF)$ the vector
$$
\mu = (1-\delta)\cdot \exp \, \rho = \Big(\, (1-\delta) \, e^{\,\rho_i} 
\, \Big)_{i\in\IN{d}}\in(\R^d)\da\ .
$$ 
Then, 
\ben
\item There exists $V_0\cdot \cF\in \cP_{(s \coma \delta)}(\cF)$ such that $\la(S_{V_0\cdot \cF})=\mu$.
\item For every $V\cdot \cF\in \cP_{(s \coma \delta)}(\cF)$
we have that 
\beq \label{desi teo frame}
\prod_{i=1}^k \mu_i 
\leq \prod_{i=1}^k\lambda_i(S_{V\cdot \cF})  \peso{for every} k\in \IN{d}  \  .
\eeq 
\item 
Equality holds in Eq. \eqref{desi teo frame} for every $k\in \IN{d}$ 
(i.e. $\la(S_{V\cdot \cF})=\mu$)
$\iff$ 
$\det(V^*V)=s$ and there exists an o.n.b. $\{v_i\}_{i\in \IN{d}}$ for $\C^d$ such that 
\beq\label{hay base 2 frame}
S_{\cF}=\sum_{i\in\IN{d}}\lambda_i(S_{\cF})\ v_i\otimes v_i \peso{and} 
VV^* 
=\sum_{i\in\IN{d}} \frac{\mu_i}{\la_i(S)}\ v_i\otimes v_i\ . 
\eeq
\een\qed
\end{teo}

\pausa
The previous result establishes the existence of perturbations $V\cdot\cF\in \cP_{(s \coma \delta)}(\cF)$ by equivalent frames, that are optimal in a rather structural sense. For example, these optimal perturbations are minimizers of convex potentials. In turn, these convex potential can be used to obtain a direct and simple (scalar) quantitative measure of performance of arbitrary perturbations within $\cP_{(s \coma \delta)}(\cF)$. We formalize these remarks in the following  

\begin{cor}\label{cor frames cercano a la identidad}
Fix an increasing function $h\in\convf$. With the notations and terminology of Theorem \ref{teo perturb multi frames}\,:
\ben
\item
For every $V\cdot \cF\in \cP_{(s \coma \delta)}(\cF)$ then
\beq \label{desi teo frames}
 P_h(V\cdot \cF)\geq  \sum_{i\in \IN{d}}h(\mu_i)\ ,
\eeq and this lower bound is attained.
\item 
Assume further that $h\in \convfs$. Then, $V\cdot \cF\in \cP_{(s \coma \delta)}(\cF)$  attains the lower bound in
\eqref{desi teo frames} iff 
$\det(V^*V)=s$ and there exists an o.n.b. $\{v_i\}_{i\in \IN{d}}$ for $\C^d$ 
such that Eq. \eqref{hay base 2 frame} holds.
\een
\end{cor}
\begin{proof}
With the notations above and using the fact that $S_{V\cdot \cF}=V^*S_\cF\,V$, then the remarks in Section \ref{subsec prelims mayo} and Theorem \ref{teo perturb multi} imply that 
$$ (\log \mu_i)_{i\in \IN{d}}\prec_w (\log (\la_i(S_{V\cdot \cF})))_{i\in\IN{d}} \implies \mu\prec_w \lambda(S_{V\cdot \cF})\ . $$
Then, the submajorization above together with the fact that $h$ is an increasing and convex function imply Eq. \eqref{desi teo frames}. Assume further that the lower bound in Eq. \eqref{desi teo frames} is attained for some $V$ as above.
Using the fact that $h$ is (an increasing) strictly convex function and the submajorization relation above then we conclude that $\mu=\lambda(S_{V\cdot \cF})$.
Hence,
by Theorem \ref{teo perturb multi}, we conclude that there exists  an o.n.b. $\{ v_i\}_{i\in \IN{d}}$ for $\C^d$ for which
Eq. \eqref{hay base 2 frame} holds. 
\end{proof}

\subsection{Expansive perturbations by equivalent frames}
\pausa
Arguing as in proof of Theorem \ref{teo perturb multi} above, in terms of Theorem \ref{teo hay max y min mayo} and
the results from \cite{MRS4} described in Remark \ref{resultados previos1}, 
we can get the following perturbation result which is of independent interest.
\begin{teo}\label{lids multi y acha}\rm
Let $S\in\matinvd^+$ and let $s>1$.
Define $$\la=(\log \, \lambda_i(S) )_{i\in\IN{d}}\in(\R^d)\da \peso{and}  t=\log \, s + \tr \, \la > \tr \, \la \ .$$ 
Let $\nu=\nu(\la\coma t)\in\Lambda_t(\la)$ be as in Remark \ref{resultados previos1}. 
Denote by $\mu = (e^{\nu_i})_{i\in\IN{d}}\,$. Then, 
\ben
\item There exists $V_0\in \matinvd$ expansive (i.e. $V_0^*V_0\geq I$) such that 
$$\det(V_0^*V_0)= s  \py \la(V_0^*S\,V_0)=\mu \ .
$$
\item If $V\in \matinvd$ is expansive and such that $\det(V^*V)= s$, then $\mu \prec_{\log} \la(V^*SV)$. 
\item Equalities hold in item 2
$\iff$ there exists an o.n.b. (of eigenvectors) 
$\{v_i\}_{i\in \IN{d}}$ for $\C^d$ such that $S$ and $VV^*$ satisfy Eq. \eqref{hay base 2}.  \QED
\een
\end{teo}

\def\expFs{\cP \cE_s(\cF)}

\pausa
We shall consider expansive perturbations of a fixed frame $\cF$ by equivalent frames, with a fixed determinant parameter: 
Let $\cF=\{f_i\}_{i\in\In}$ be a frame for $\C^d$ and let $s>1$. Denote by 
$$
\expFs = \big\{V\cdot \cF : V\in \matinvd \ , \ \ \det \, V^*V = s \py V^*V\ge I\big\} \ .
$$
\begin{cor}\label{teo perturb multi frames expansive}\rm
Let $\cF=\{f_i\}_{i\in\In}$ be a frame for $\C^d$ with frame operator $S_\cF\in\matinvd^+$. 
Let  $s>1$. 
Consider the data $\la \coma t\coma \nu$ and $ \mu$ as in Theorem \ref{lids multi y acha} with $S = S_\cF\,$. 
Then, 
\ben
\item There exists $V_0\cdot \cF\in \expFs$ such that $\la(S_{V_0\cdot \cF})=\mu$.
\item For every other $V\cdot \cF\in \expFs
$
we have that 
$\prodl_{i=1}^k \mu_i 
\leq \prodl_{i=1}^k\lambda_i(S_{V\cdot \cF})$ for $k\in \IN{d}\,$.
\item 
Equality holds in item 2 
for every $k\in \IN{d}$ 
(i.e. $\la(S_{V\cdot \cF})=\mu$)
$\iff$ 
there exists an o.n.b. $\{v_i\}_{i\in \IN{d}}$ for $\C^d$ such that 
$S_\cF$ and $VV^*$ satisfy Eq. \eqref{hay base 2 frame}. \QED
\een
\end{cor}

\section{Appendix - multiplicative Lidskii's inequality}\label{sec append}

Although majorization and log-majorization are not a total relations in $\R^d$ they appear naturally in many situations in matrix analysis. Some examples of this phenomenon are the Schur-Horn theorem characterizing the main diagonals of unitary conjugates of a selfadjoint matrix $A$, and Horn's relations between the eigenvalues and singular values of matrices. Another interesting example of a majorization relation is Lidskii's inequality for selfadjoint matrices; namely, if  $A,\,B$ are selfadjoint matrices then $\lambda(A)+\lambda^\uparrow(B)\prec \lambda(A+B)$. Lidskii's inequality has a multiplicative version obtained by Li and Mathias in \cite{LiMat}. In the positive invertible case, Li and Mathias's results can be put into the deep theory of singular value inequalities developed by Klyachko in \cite{Klya}. Next, we describe these results in detail and characterize the case of equality in Li-Mathias's multiplicative inequality.
The following inequalities are the multiplicative version of Lidskii's inequality: 

\begin{teo}[\cite{LiMat}]\label{LiMat}\rm 
Let $S\in\matpos$ and $V\in \matinvd$. Let $J\inc \IN{d}\,$ be such that  $|J|=k$ and 
$\lambda_i(S)>0$ for $i\in J$. Then we have that 
\begin{equation}
\prod_{i=1}^k \lambda_{d+1-i}(V^*V)\leq \prod_{i\in J} \frac{\lambda_i(V^*SV)}{\lambda_i(S)}\leq \prod_{i=1}^k\lambda_i(V^*V)\ . 
\QEDP
\end{equation}
\end{teo}

\pausa
We complement this result by characterizing the case of equality in the inequalities above. 
We shall use some results of \cite [Section 8]{MRS12}, where we study the case of equality in the 
additive Lidskii's inequalities, and also 
some combinatorial problems. 
We begin by revisiting the following well known inequality from matrix theory. Our interest relies in the case of equality.

\begin{lem}[Weyl's inequalities] \label{lemWineq}
Let $A,\,B \in \mat$ be  Hermitian matrices. Then,
\beq\label{eqW1}
\lambda_i(A+B)\leq \lambda_i(A)+\lambda_{1}(B) \peso{for every} i\in\IN{d} \ . 
\eeq
If there is $i\in \IN{d}$ such that 
$\lambda_i(A+B)=\lambda_i(A)+\lambda_{1}(B) $ then there is a unit vector $x$ such that 
$$ 
A\,x=\lambda_i(A)\,x \py B\,x=\lambda_{1}(B)\,x \ .$$
\end{lem}
\proof It is a particular case of \cite[Lemma 8.1]{MRS12}. \QED

\begin{pro}[Ostrowski's inequality] \label{ostrowski}
Let $S\in\matpos$ and let $V\in \mat$ be such that $V^*V\geq I$. Then, for $i\in \IN{d}$ we have that  
\begin{equation} \lambda_i(S)\leq \lambda_i(V^*SV) \ .\end{equation} Moreover, if there exists $J\subset \IN{d}$ such that $\lambda_i(S)=\lambda_i(V^*SV)$ for $i\in J$, then there exists an o.n.s. $\{v_i\}_{i\in I}\subset \C^d$ such that $$|V^*|\ v_i=v_i \peso{and}  S\ v_i=\lambda_i(S)\,v_i \peso{for} i\in I\ .  $$
\end{pro}
\begin{proof} 
The first part of the statement is well known (see for example \cite[Thm. 5.4.9.]{Horn}).
Hence, we prove the second part of the statement by induction on $|J|$, 
the number of elements of $J$. 
Assume first that $V\in\matpos$ is such that $V^2\geq I$. Fix $i\in J$ and 
notice that, by Sylvester's law of inertia, $\lambda_i(V\,(S-\lambda_i(S)\, I)\,V)=0$, 
since $\lambda_i(S-\lambda_i(S)\, I)=0$. By Lemma 
\ref {lemWineq} 
we have that 
\beq\label{lad=1}
\lambda_i(VSV)- \lambda_i(S) \lambda_d(V^2)
= \lambda_i(VSV)+\lambda_1(-\lambda_i(S)\,V^2)
\stackrel{\ref{lemWineq}}{\geq} \lambda_i(V\,(S-\lambda_i(S)\, I)\,V)=0\ . 
\eeq
Since $\lambda_d(V^2)\geq 1$ and $\lambda_i(VSV)=\lambda_i(S)$ ( $i\in J$ ) we conclude that $\lambda_d(V^2)=1$. Moreover, by the equality in Eq. \eqref{lad=1} and Lemma \ref{lemWineq}, 
there is $x\in\C^d$, $\|x\|=1$ such that 
$ VSV x \stackrel{\ref{lemWineq}}{=}\lambda_i(VSV)\ x 
$ and
$$-\lambda_i(S) V^2x 
\stackrel{\ref{lemWineq}}{=} \lambda_1(-\lambda_i(S) V^2) \ x 
= -\lambda_i(S) \,\lambda_d(V^2)\ x
=-\lambda_i(S) \ x \ .
$$
Hence $V^2x=x$ and then $Vx=x$. Thus, 
$
VSV x = \lambda_i(VSV)\ x \implies
S x=\lambda_i(VSV) \ x=\lambda_i(S)\ x\ .
$

\pausa 
This proves the statement for $|J|=1$. If we assume that $|J|>1$ then we fix $v_i=x$ and consider $\mathcal W=\{v_i\}^\perp$, which reduces both  $A$ and $V$. 
Therefore an easy inductive argument involving the restriction 
$S|_\mathcal W$ and $V|_\mathcal W$ shows the general case. 

\pausa
If we now consider an arbitrary $V\in\mat$ such that $V^*V\geq I$ then let $V^*=U\,|V^*|$ 
be the polar decomposition of $V^*$. In this case $V^*SV=U |V^*|\ S\ |V^*| U^*$ so 
that $\lambda_i(V^*SV)=\lambda_i(|V^*|\ S \ |V^*|)$ for every $i\in \IN{d}\,$, 
where $|V^*|^2=VV^*=U^*(V^*V)U\geq I$. These last facts together with the case 
of equality for positive expansions prove the statement.
\end{proof}

\pausa

\pausa
In order to state our results we introduce the following notion.

\begin{fed}\label{multiplicative matching} \rm 
Let $S\in\matinvd^+$ and let $V\in\matinvd$. We say that $V$ is an upper 
multiplicative matching (UMM) of $S$ (resp. lower MM or LMM of $S$) 
if there exists a family $\{J_k\}_{k\in\IN{d}}$ 
such that $J_k\inc J_{k+1}\inc \IN{d}$ for $1\leq k\leq d-1$, $|J_k| =k$ for 
$k\in \IN{d}$ and such that 
$$ 
\prod_{i\in J_k} \frac{\lambda_i(V^*SV)}{\lambda_i(S)}=\prod_{i=1}^k\lambda_i(V^*V) 
\ , \  \ \ k\in \IN{d}
$$ 
\Big(resp. $\prod_{i=1}^k\lambda_{i}\ua(V^*V) 
= \prod_{i=1}^k\lambda_{d+1-i}(V^*V) 
= \prod_{i\in J_k} \frac{\lambda_i(V^*SV)}{\lambda_i(S)} \,$, $k\in \IN{d}$\Big).
\EOE
\end{fed}

\begin{teo}\label{teo carac Li-Mat}
Let $S\in\matinvd^+$ and let $V\in\matinvd$ be 
a UMM or a LMM of $S$. Then $S$ and $|V^*|$ commute.
\end{teo}
\begin{proof} 
We can assume that $V$ is not a multiple of the identity. 
We use the splitting technique considered in \cite{LiMat}.
Let $V\in\matinvd$ be an UMM of $S$. 
Assume further that $V\in\matinvd^+$ and let $\{u_i\}_{i\in\IN{d}}\,$ be a o.n.b for $\C^d$ such that 
$V= \suml_{i\in\IN{d}} \la_i(V) \, u_i\otimes u_i\,$. 
Let $2\leq k\leq {d}$ be such that $\lambda_{k-1}>\lambda_k\,$.  
Let $V_k=\lambda_k^{-1}\,V$, which is also an UMM for $S$. 
In this case $\lambda_i(V_k)=\frac{\lambda_i(V)}{\lambda_k(V)}$ for 
every $i\in \IN{d}\,$. In particular, $\lambda_k(V_k)=1$.
Let $B_k= \suml_{i\in\IN{d}} \mu_i \,\cdot  \, u_i\otimes u_i\in\matinvd^+\,$, 
where $$
\mu
=(\lambda_1(V_k)\coma \ldots\coma \lambda_{k-1}(V_k)\coma \uno_{d-k} 
)\in (\R_{>0}^d)\da \  .
$$
Notice that $W_k=\ker(B_k-I) =\gen\{u_i: k\le i \le d\} \implies 
\dim \,W_k=d+1-k$. Also notice that the orthogonal projection $P_k \igdef P_{W_k}$  
 coincides with the spectral projection of $V$ corresponding 
to the interval $(0,\lambda_k(V)]$.

\pausa
On the other hand, by construction of $B_k\,$, we see that $B_k\geq I$ and 
$V_k^{-1}B_k=B_kV_k^{-1}\geq I$. Using that $V_k$ is an UMM of $S$,  we can take $J_{k-1}\inc\IN{d}\,$ 
such that $|J_{k-1}|=k-1$  and 
\begin{equation}\label{hip1} \prod_{i\in J_{k-1}} \frac{\lambda_i(V_kSV_k)}{\lambda_i(S)}=\prod_{i=1}^{k-1}\lambda_i(V_k^2)\ .\end{equation}
Then, by Ostrowski's inequality we get that 
$$ \prod_{i\in J_{k-1}} \frac{\lambda_i(V_k S\, V_k)}{\lambda_i(S)} \leq \prod_{i\in J_{k-1}} \frac{\lambda_i((B_k\,V_k^{-1})(V_kS\,V_k)(V_k^{-1}\,B_k))}{\lambda_i(S)}=\prod_{i\in J_{k-1}}\frac{\lambda_i(B_k S \,B_k)}{\lambda_i(S)}\ . $$ Using Ostrowski's inequality again, we see that $\frac{\lambda_i(B_k S\,B_k)}{\lambda_i(S)}\geq 1$ for every $i\in\IN{d}$ and therefore
\begin{eqnarray*}
\prod_{i\in J_{k-1}} \frac{\lambda_i(V_kS\,V_k)}{\lambda_i(S)}&\leq&  
\prod_{i\in J_{k-1}}\frac{\lambda_i(B_k S \,B_k)}{\lambda_i(S)} 
\leq \prod_{i\in\IN{d}}\frac{\lambda_i(B_k S \,B_k)}{\lambda_i(S)}\\
&=&\frac{\det(B_kS\,B_k)}{\det(S)}=\det(B_k^2)=\prod_{i=1}^{k-1} \lambda_i(V_k^2) \ .
\end{eqnarray*} 
By Eq. \eqref{hip1} we see that the previous inequalities are actually equalities.
Hence, if we let $J_{k-1}^c=\IN{d}\setminus J_{k-1}$ then $|J_{k-1}^c|=d+1-k$ and 
$$
\prod_{i\in J_{k-1}^c}\frac{\lambda_i(B_kS\,B_k)}{\lambda_i(S)}=1\implies 
 \lambda_i(B_kS\,B_k)= \lambda_i(S) \peso{for} i\in J_{k-1}^c\ .
$$
By the case of equality in Ostrowski's inequality in Proposition \ref{ostrowski} we conclude that there exists an o.n.s. $\{v_i\}_{i\in J_{k-1}^c}\inc \C^d$ such that 
\beq\label{aaa}
B_kv_i=v_i \peso{and} Sv_i=\lambda_i(S) \peso{for} i\in J_{k-1}^c \ .
\eeq
Then we conclude that $\{v_i\}_{i\in J_{k-1}^c}$ is another o.n.b. of $W_k\,$. Hence 
$P_k=\sum_{i\in J_{k-1}^c} v_i\otimes v_i\,$ and, by Eq. \eqref{aaa}, 
we conclude that $P_k$ and $S$ commute. 
Finally, since $V$ can be written as a linear combination of its spectral projections 
$P_k$  (for $\lambda_{k-1}>\lambda_k\,$) and the identity $I$, we see that $V$ and $S$ commute in this case.
The general case for arbitrary $V\in\matinvd$ follows from the positive case with the 
reduction described at the end of the proof of Proposition \ref{ostrowski}.

\pausa
Assume  now that $V\in\matpos$
is a LMM of $S$.
Then $V^{-1}$ is an UMM for $VSV$. Indeed, if $J_k\inc \IN{d}$ is such that $$ 
\prod_{i=1}^k\lambda_{i}\ua(V^2)= 
\prod_{i\in J_k} \frac{\lambda_i(VSV)}{\lambda_i(S)} \ , 
$$ 
then we have that
$$ 
\prod_{i=1}^k \lambda_i(V^{-2})=
\left( \prod_{i=1}^k\lambda_{i}\ua(V^2)
\right)^{-1} =\left( \prod_{i\in J_k} \frac{\lambda_i(VSV)}{\lambda_i(S)}\right)^{-1} 
=\prod_{i\in J_k} \frac{\lambda_i(V^{-1}(VSV)\,V^{-1})}{\lambda_i(VSV)}\ .
$$ 
By the first part of this proof, we conclude that $V^{-1}$ and $VSV$ commute, which implies that $S$ and $V$ commute. If $V\in\matinvd$ is arbitrary we conclude that $S$ and $|V^*|$ commute 
with the reduction described at the end of the proof of Proposition \ref{ostrowski}.
\end{proof}

\pausa
Now we can re-state and prove Theorem \ref {teo hay max y min mayo}: 

\pausa
{\bf Theorem \ref{teo hay max y min mayo}} 
Let $S\in\matinvd^+$ and let $\ga\in (\R_{>0}^d)\da$. 
Then, for every $V\in\matinvd$ such that $\lambda(V^*V)=\ga$ we have that 
$$
\lambda(S)\circ \ga\ua\prec_{\log}\la(V^*SV)\prec_{\log}\lambda(S)\circ \ga
\in (\R_{>0}^d)\da \ .
$$
Moreover, if $\lambda(V^*SV)=(\lambda(S)\circ \ga\ua)\da$ 
(resp. 
$\lambda(V^*SV)=\lambda(S)\circ \ga$) then there exists an o.n.b. 
$\{ v_i\}_{i\in \IN{d}}$ of $\C^d$ such that 
\beq\label{hay base bis}
S=\sum_{i\in \IN{d}} \la_i(S)\ v_i\otimes v_i \peso{and} |V^*|=
\sum_{i\in \IN{d}} \ga_{d+1-i}\rai\ v_i\otimes v_i\  
\eeq 
\big(\, resp. $S=\sum_{i\in \IN{d}} \lambda_i(S)\ v_i\otimes v_i$ and 
$|V^*|=\sum_{i\in \IN{d}} \ga_i\rai 
\ v_i\otimes v_i$\,\big).

\proof 
Let $S$ and $V$ be as above. Assume further that $V\in\matinvd^+$ and notice that then $\lambda(V\,S\,V)=\lambda(S^{1/2}V^2\, S^{1/2})$. By Theorem \ref{LiMat} we get that, 
for every $J\subset \IN{d}$ with $|J|=k$,  
$$
\prod_{i\in J}\lambda_{i}\ua(S) \ \lambda_{i}(V^2)=
\prod_{i\in J}\frac{\lambda_i(S^{-1/2}(S^{1/2}V^2 S^{1/2})S^{-1/2})}{\lambda_i(S^{-1})}\leq\prod_{i=1}^k\lambda_i(S^{1/2}V^2S^{1/2})\ .
$$
This 
shows that $\lambda\circ\lambda\ua(S)\prec_{\log}\lambda(V\,S\,V)$ 
or equivalently, that $\lambda(S)\circ \la\ua\prec_{\log}\lambda(V\,S\,V)$. 
Moreover, the previous facts also show that if 
$\lambda(V\,S\,V)=(\lambda(S)\circ \lambda\ua)\da$ then $S^{-1/2}$ is an 
UMM of $S^{1/2}V^2S^{1/2}$. By Theorem \ref{teo carac Li-Mat} 
we see that  $S^{-1/2}$ and $S^{1/2}V^2 S^{1/2}$ commute, 
 which in turn implies that $S$ and $V$ commute.

\pausa
Since $S$ and $V$ commute we conclude that there exists an o.n.b. 
$\{w_i\}_{i\in\IN{d}}$ of $\C^d$ such that 
$$ 
S=\sum_{i\in\IN{d}} \la_i(S)\ w_i\otimes w_i \py
V=\sum_{i\in\IN{d}} \la\ua_{\sigma(i)}(V)\ w_i\otimes w_i 
$$ 
for some permutation
$\sigma\in\mathbb S_d\,$. That is, in this case we have that 
$$
\Big(\,\lambda(S)\circ \lambda\ua(V^2)\,\Big)\da=\lambda(VSV)
=\Big(\, \lambda(S)\circ \lambda\ua_\sigma(V^2) \, \Big)\da .
$$
Notice that by replacing $S$ and $V$ by $tS$ and $tV$ for sufficiently large $t>0$ 
we can always assume that $S-I\in\matinvd^+$ and $V-I\in\matinvd^+$. 
Using the properties of the logarithm, 
we conclude that the vectors $\log \, \lambda(S)$ and $ \log \, \lambda(V^2)\in (\R_{>0})\da$ 
are such that 
$$
\Big(\, \log \, \lambda(S)+ \log \, \lambda\ua(V^2)\, \Big)\da
=\Big(\, \log \, \lambda(S) + \log \, \lambda_\sigma\ua(V^2) \, \Big)\da \ . 
$$
By \cite[Proposition 8.6 and Remark 8.7]{MRS12} 
we conclude 
that $\log \, \lambda(S) =\log \, \lambda_\sigma(S)$. That is, if we set $v_i=w_{\sigma^{-1}(i)}$ for $i\in \IN{d}$ then
the o.n.b. $\{v_i\}_{i\in \IN{d}}$ satisfies the conditions in Eq. \eqref{hay base bis}.
The general case, for $V\in\matinvd$, follows by the reduction described at the end of the proof of Proposition \ref{ostrowski}.

\pausa
On the other hand, notice that a direct application of Theorem \ref{LiMat} shows that 
$$ 
\prod_{i=1}^k \frac{\la_i(V^*SV)}{\la_i(S)}\leq \prod_{i=1}^k\la_i(V^*V)\implies
\prod_{i=1}^k \la_i(V^*SV)\leq \prod_{i=1}^k {\la_i(S)} \ \la_i(V^*V) \ . 
$$ 
Hence, we conclude that $\la(V^*SV)\prec_{\log}\la(S)\circ \la(V^*V)\in (\R^d_{>0})\da$. 
Finally, in case that  $\la(V^*SV)=\la(S)\circ \la(V^*V)$ we see that $S$ is an 
UMM for $S$ and therefore $S$ and $|V^*|$ commute. 
In this case it is straightforward to check that there exists an o.n.b. 
$\{v_i\}_{i\in \IN{d}}$ with the desired properties.
\QED

\fontsize {8}{9}\selectfont


\begin{thebibliography}{99}

\bibitem{Bal} R. Balan, Equivalence relations and distances between Hilbert frames. Proc. Amer. Math. Soc. 127 (1999), no. 8, 2353-2366.

\bibitem{Bat} R. Bhatia,  Matrix Analysis,
Berlin-Heildelberg-New York, Springer 1997.

\bibitem{BF} J.J. Benedetto, M. Fickus,  Finite normalized tight frames, Adv. Comput. Math. 18, No. 2-4 (2003), 357-385 .
\bibitem {Horn} R.A. Horn, C.R. Johnson, Matrix analysis. Second edition. Cambridge University Press, Cambridge, 2013.



\bibitem{CKFT} P.G. Casazza, M. Fickus, J. Kovacevic, M.T. Leon, J.C. Tremain, A physical interpretation of tight frames. Harmonic analysis and applications, 51--76, Appl. Numer. Harmon. Anal., Birkhäuser Boston, MA, 2006.

\bibitem{FF} Eds. P. G. Casazza and G. Kutyniok, Finite Frames: Theory and Applications,  Applied and Numerical Harmonic Analysis. Birkhäuser/Springer, New York, 2013.


\bibitem{Chrisbook} O. Christensen, An introduction to frames and Riesz bases. Applied and Numerical Harmonic Analysis. Birkhäuser Boston, Inc., Boston, MA, 2003.

\bibitem{FMP} M. Fickus, D. G. Mixon and M. J. Poteet,  Frame completions for
optimally robust reconstruction, Proceedings of SPIE, 8138: 81380Q/1-8 (2011).



\bibitem{Han} D. Han, Frame representations and Parseval duals with applications to Gabor frames. Trans. Amer. Math. Soc. 360 (2008), no. 6, 3307-3326.


\bibitem{Klya}A.A. Klyachko,  Random walks on symmetric spaces and inequalities for matrix spectra. Special Issue: Workshop on Geometric and Combinatorial Methods in the Hermitian Sum Spectral Problem (Coimbra, 1999). Linear Algebra Appl. 319 (2000), no. 1-3, 37-59.

\bibitem{LeHanagre}J. Leng, D. Han, Optimal dual frames for erasures II. Linear Algebra Appl. 435 (2011), 1464-1472. 


\bibitem{LiMat} C.K. Li, R. Mathias, The Lidskii-Mirsky-Wielandt theorem - additive and multiplicative versions. Numer. Math. 81 (1999), no. 3, 377-413.

\bibitem{LoHanagre}J. Lopez, D. Han, Optimal dual frames for erasures. Linear Algebra Appl. 432 (2010), 471-482.



\bibitem{MR} P. Massey and M. Ruiz,  Minimization of convex functionals over frame operators, Adv.  Comput. Math. 32 (2010),  131-153.


\bibitem{MRS2} P. Massey, M. Ruiz and D. Stojanoff, Duality in  reconstruction systems. Linear Algebra Appl. 436 (2012), 447-464.

\bibitem{MRS3} P. Massey, M. Ruiz and D. Stojanoff,  
Robust Dual Reconstruction Systems and Fusion Frames. Acta Appl. Math. 119 (2012), 167-183.

\bibitem{MRS4} P. Massey, M. Ruiz and D. Stojanoff,  Optimal dual frames and frame completions for majorization, Appl. Comput. Harmon. Anal. 34 (2013), no. 2, 201-223.

\bibitem{MRS12} P. Massey, M. Ruiz and D. Stojanoff, Optimal frame completions, Adv. Comput. Math, DOI 10.1007/s10444-013-9339-7.
\bibitem{MRS pre} P. Massey, M. Ruiz and D. Stojanoff,
Optimal frame completions with prescribed norms for majorization, submitted.
 

\end{thebibliography}
\end{document}